\documentclass[11pt, twoside]{article}

\usepackage[a4paper,left=2.5cm,right=2.5cm,top=2.5cm,bottom=2.5cm]{geometry} 
\usepackage{enumerate}
\usepackage{amsmath,amssymb,amsthm,amsfonts,dsfont,mathrsfs}
\usepackage{bbm}
\usepackage{setspace}
\usepackage{titlesec}
\usepackage{graphicx, fancyhdr}
\usepackage{array}
\usepackage{tabularx,ragged2e,booktabs}
\usepackage{algorithm}
\usepackage[noend]{algpseudocode}
\usepackage{caption}
\usepackage[margin=0.1cm]{subcaption}
\usepackage{enumitem}
\usepackage{soul}
\usepackage{footnote}
\usepackage[dvipsnames]{xcolor}
\usepackage{comment}
\usepackage{moresize} 
\usepackage{mdframed}
\usepackage{multirow}
\usepackage{relsize}

\usepackage{hyperref} 
\usepackage[nameinlink]{cleveref}
\let\cref = \Cref
\hypersetup{colorlinks=true,
            citecolor=ForestGreen,
            linkcolor=ForestGreen,    
            urlcolor=Brown}
\crefname{subsection}{Section}{Sections}
\crefname{section}{Section}{Sections}

\usepackage[backend=biber,style=numeric]{biblatex}
\newtheorem{theorem}{Theorem}[section]
\newtheorem{lemma}[theorem]{Lemma}
\newtheorem{proposition}[theorem]{Proposition}
\newtheorem{corollary}[theorem]{Corollary}

\newtheorem{example}[theorem]{Example}
\newtheorem{remark}[theorem]{Remark}
\newtheorem{definition}[theorem]{Definition}

\renewcommand{\leq}{\leqslant} 
\renewcommand{\geq}{\geqslant} 

\newcommand{\ra}{\rangle}
\newcommand{\la}{\langle}

\newcommand{\ind}[1]{\mathds{1}_{#1}}
\newcommand{\eps}{\varepsilon}

\newcommand{\norm}[1]{\left\Vert#1\right\Vert}

\newcommand{\set}[1]{\left\{#1\right\}}

\newcommand{\ALS}{$\textup{ALS}^{++}$~}

 \let\gb=\beta \let\gc=\gamma  \let\gee=\varepsilon
     \let\gl=\lambda       
   \let\gr=\rho \let\gs=\sigma \let\gt=\tau 
  
\let\gC=\Gamma \let\gD=\Delta   
   \let\gP=\Pi      \let\gS=\Sigma

\newcommand{\cB}{\mathcal{B}}
\newcommand{\cD}{\mathcal{D}}

\newcommand{\cM}{\mathcal{M}}\newcommand{\cO}{\mathcal{O}}

\newcommand{\bs}[1]{\boldsymbol{#1}}

\newcommand{\sC}{\mathscr{C}}

\newcommand{\sR}{\mathscr{R}}

\DeclareMathOperator{\E}{\mathds{E}}
\DeclareMathOperator{\R}{\mathbb{R}}

\DeclareMathOperator{\pr}{\mathds{P}}

\DeclareMathOperator{\var}{Var}

\DeclareMathOperator*{\argmax}{argmax}
\DeclareMathOperator*{\argmin}{argmin}

\DeclareMathOperator{\sym}{sym}

\let\P = \bsP

\let\a = \bsa

\DeclareMathOperator{\rank}{rank}

\DeclareMathOperator{\prepnormeqn}{\textsc{PrepNormEqn}}
\DeclareMathOperator{\wstep}{\textsc{UpdateWeight}}
\DeclareMathOperator{\mstep}{\textsc{UpdateMean}}

\DeclareMathOperator{\solverow}{\textsc{SolveRow}}

\newcommand{\Diag}[1]{\textbf{Diag}\left(#1\right)}

\newcommand{\defil}[1]{\textit{#1}}  

\renewcommand{\v}[1]{\bs{#1}} 

\renewcommand{\t}[1]{\bs{\mathcal{#1}}} 

\newcommand{\m}[1]{\bs{#1}}

\newcommand{\tprod}[2]{{#1}^{\otimes #2}}  
\newcommand{\eprod}[2]{{#1}^{\ast #2}}

\newcommand{\floor}[1]{\left\lfloor #1 \right\rfloor}
\newcommand{\ceil}[1]{\left\lceil #1 \right\rceil}

\newcommand{\ip}[2]{\left\la #1,\, #2 \right\ra}
\newcommand{\pd}[2]{\frac{\partial #1}{\partial #2}}

\def\pom{\textit{pomegranate }}

\date{}
\title{Moment Estimation of Nonparametric Mixture Models Through Implicit Tensor Decomposition\thanks{Both authors were supported in part by start-up grants to JK provided by the College of Natural Sciences and Oden Institute at UT Austin.  We thank Jonathan Niles-Weed for helpful conversations.}}

\author{Yifan Zhang\thanks{Oden Institute, University of Texas at Austin, yf.zhang@utexas.edu}
\and Joe Kileel\thanks{Department of Mathematics and Oden Institute, University of Texas at Austin, jkileel@math.utexas.edu}}

\begin{document}
    \maketitle
    \begin{abstract}   
        We present an alternating least squares type numerical optimization scheme to estimate conditionally-independent mixture models in $\mathbb{R}^n$, without parameterizing the distributions. 
        Following the method of moments, we tackle an incomplete tensor decomposition problem to learn the mixing weights and componentwise means.
        Then we compute the cumulative distribution functions, higher moments and other statistics of the component distributions through linear solves.
        Crucially for computations in high dimensions, the steep costs associated with high-order tensors are evaded, via the development of efficient tensor-free operations.
        Numerical experiments demonstrate the competitive performance of the algorithm, and its applicability to many models and applications.
        Furthermore we provide theoretical analyses, establishing identifiability from low-order moments of the mixture and guaranteeing local linear convergence of the ALS algorithm.
    \end{abstract}

    \section{Introduction} \label{sec:intro}

Mixture models have been intensively studied for their strong expressive power \cite{lindsey1995modelling,bohning2003recent,mclachlan2004finite}. 
They are a common choice in data science problems when data is believed to come from distinct subgroups, or when simple distributions such as one Gaussian or one Poisson do not offer a good fit on their own.

In this paper, we develop new numerical tools for the nonparametric estimation of 
(finite) \textup{conditionally-independent mixture models} in \nolinebreak $\mathbb{R}^n$, also known as mixtures of product distributions: 
\begin{equation}\label{eq:cond-inp-mix}
    \mathcal{D} = \sum\nolimits_{j=1}^r w_j \cD_j~~\text{where}~~\cD_j = \bigotimes \nolimits_{i=1}^n \mathcal{D}_{ij}.
\end{equation}
Here $r$ is the number of components, and
each $\cD_j$ is a product of $n$ distributions $\mathcal{D}_{ij}$ ($i=1, \ldots, n$) on the real line.
No parametric assumptions are placed on $\mathcal{D}_{ij}$ except that their moments should exist.
Our new methods are based on tensor decomposition. 
Our approach is particularly well-suited for computations in high dimensions.

\subsection{Why This Model?} 
The nonparametric conditionally-independent mixture model \eqref{eq:cond-inp-mix} was brought to broader attention by Hall and Zhou \cite{hall2003nonparametric}.
Their motivation came from clinical applications, modeling heterogeneity of repeated test data from patient groups with different diseases.
Later the model was used to describe heterogeneity in econometrics by Kasahara and Shimotsu \cite{kasahara2009nonparametric}, who considered dynamic discrete choice models.  
By modeling a choice policy as a mixture of Markovian policies, they reduced their problem to an instance of \eqref{eq:cond-inp-mix}.
Heterogeneity analysis using \eqref{eq:cond-inp-mix} also appears in auction analysis \cite{hu2013identification,hu2012nonparametric} and incomplete information games \cite{xiao2018identification,aguirregabiria2019identification}; see the survey  \cite{chauveau2015semi}.  
In statistics, Bonhomme and coworkers \cite{bonhomme2016non} considered repeatedly measuring a real-valued mixture and reduced that situation to instance of \eqref{eq:cond-inp-mix}.
Extending their framework, conditionally-independent mixtures can model repeated measurements of $r$ deterministic objects described by $n$ features, when there is independent measurement noise for each feature.  This setting approximates the class averaging problem of cryo-electron microscopy \cite{bhamre2017mahalanobis}.
In computer science, Jain and Oh studied \eqref{eq:cond-inp-mix} when $\cD_{ij}$ are finite discrete distributions, motivated by connections to crowdsourcing, genetics, and recommendation systems.
Chen and Moitra \cite{chen2019beyond} studied \eqref{eq:cond-inp-mix} to learn stochastic decision trees, where each $\cD_{ij}$ are Bernoulli variables.

Conditional independence can also be viewed as a reduced model or regularization of the full model.
Here \eqref{eq:cond-inp-mix} is a universal approximator of densities on $\mathbb{R}^n$. 
But also, it reduces the number of parameters and prevents possible over-fitting. 
For instance, experiments in \cite{kargas2019learning} show that for many real datasets, 
diagonal Gaussian mixtures give a comparable or better classification accuracy compared to full Gaussian mixtures.

\subsection{Existing Methods}
Many approaches for estimating model \eqref{eq:cond-inp-mix} have been studied.
Hall and Zhou \cite{hall2003nonparametric} took a nonparametric maximum-likelihood-based approach to learn the component distributions and weights when $r=2$.
Subsequently in \cite{hall2005nonparametric}, Hall et al. provided another consistent density estimator for \eqref{eq:cond-inp-mix} when $r = 2$, with  errors bounded by $\mathcal{O}(p^{-1/2})$.
A semi-parametric expectation maximization framework has been developed in a series of works \cite{benaglia2009EM,levine2011maximum,chauveau2016nonparametric}, where $\cD_j$ is a superposition of parametrized kernel functions.
Bonhomme and coauthors \cite{bonhomme2016non} pursued a different approach that applies to repeated 1-D measurements. 
 They learned the map $\phi \mapsto \E_{X\sim \cD_j} [\phi(X)]$, instead of the distributions $\mathcal{D}_j$ directly.  

Tensor decomposition and method of moment (MoM) based algorithms for \eqref{eq:cond-inp-mix} have also received  attention.
Jain and Oh \cite{jain2014learning} as well as Chen and Moitra \cite{chen2019beyond} developed moment-based approaches to learn discrete mixtures.
\cite{anandkumar2012method} is another example of a MoM approach for \eqref{eq:cond-inp-mix} in the discrete case.
For continuous distributions and MoM, most attention was devoted to solving diagonal Gaussian mixtures \cite{hsu2013learning,guo2022learning,khouja2022tensor}.
For nonparametric estimation, the works \cite{kargas2019learning} and  \cite{zheng2020nonparametric} used tensor techniques to learn discretizations of $\mathcal{D}_j$.
In terms of computational advances, Sherman and Kolda \cite{sherman2020estimating} proposed implicit tensor decomposition to avoid expensive tensor formation when applying MoM.
Subsequently, Pereira and coauthors \cite{pereira2022tensor} developed an implicit MoM algorithm for diagonal and full Gaussian mixture models.

Regarding theory, identifiability has been investigated.
This includes identifiability of the number of components ($r$), mixing weights, and component distributions.
Allman and coauthors \cite{allman_identifiability_2009} proved that mixing weights and component distributions can be identified from the \emph{joint distribution}, assuming $r$ is known.
In the large sample limit, this implies identifiability from \emph{data}.
For learning $r$, recent works are  \cite{kasahara2014non, kwon2021estimation}.

\subsection{Our Contributions} \label{sec:contrib}

We learn the nonparametric mixture model \eqref{eq:cond-inp-mix} through the method of moments.
The number of components $r$ is assumed to be known, or already estimated by an existing method.
Given a dataset $\m V = \{\v v_1, \ldots, \v v_p \}$ of size $p$ in $\R^n$ sampled independently from $\mathcal{D}$,
we learn the weights $\v w = \set{w_1,\ldots,w_r}$ and a functional similar to the one in Bonhomme et al.'s paper \cite{bonhomme2016non}:
\begin{equation} \label{eq:fnal}
    g \mapsto \E_{X \sim \cD_j}[g(X)], \ \ \text{where}\ \ g(\v x) = (g_1(x_1),\ldots,g_n(x_n)),\footnote{The functional applies to functions $g : \mathbb{R}^n \rightarrow \mathbb{R}^n$ that are Cartesian products of  $g_i : \mathbb{R} \rightarrow \mathbb{R}$.}
\end{equation}  
for all $j \in [r]$.
We call the output of the functional a \textit{general mean}.
In particular, this includes the component distributions ($g_i(x_i) = \ind{x_i \leq t_i}$), component moments ($g_i(x_i) = x_i^{m_i}$), moment generating functions ($g_i(x_i) = e^{t_i x_i}$), etc.

We propose a novel two-step approach.
First ({step 1}), we learn the mixing weights and component means using \textup{tensor decomposition} and an alternating least squares (ALS) type algorithm (\cref{alg:basicALS} and \ref{alg:finalALS}).  
This exploits special algebraic structure in the conditional-independent model.
Once the weights and means are learned, ({step 2}) we make an important observation that for any $g$ in \eqref{eq:fnal}, the general means $\E_{X \sim \cD_j} [g(X)]$  ($j = 1,\ldots,r$) are computed 
by a linear solve (\cref{alg:gmom})

We pay great attention to the efficiency and \textup{scalability} of our MoM based algorithm, particularly for high dimensions.
Our MoM algorithm is \textit{implicit}  in the sense of \cite{sherman2020estimating}, meaning that no moment tensor is explicitly computed during the estimation.
For model \eqref{eq:cond-inp-mix} in $\R^n$ of sample size $p$,
this allows us to achieve an $\cO(npr + nr^3)$ computational complexity to update the means and weights in ALS in step 1, similar to the per-iteration cost of the EM algorithms.
The storage cost is only $\cO(n(r + p))$ -- the same order needed to store the data and computed means.
This is a great improvement compared to explicitly computing moment tensors, which requires $\cO(pn^d)$ computation and $\cO(n^d)$ storage for the $d$th moment.  The cost to solve for the general means (step 2) is $\cO(npr + nr^3)$ time.

Compared to other recent tensor-based methods \cite{zheng2020nonparametric,kargas2019learning}, it may seem indirect to learn the functional \eqref{eq:fnal} rather than learning the component distributions directly.
However our approach has major advantages for large $n$.
In \cite{kargas2019learning}, the authors need CP decompositions of 3-D histograms for all marginalizations of the data onto $3$ features.  
This scales cubically in $n$.  
Additionally, the needed resolution of the histograms can be hard to estimate a priori.
In \cite{zheng2020nonparametric}, the authors learn the densities by running kernel density estimation.  Then they compute a CP decomposition of an explicit $r^n$ tensor, which is exponential in $n$.
By contrast, our algorithm has much lower complexity. 
Moreover, the functional-based approach enables adaptive localization of the component distributions by estimating e.g. their means $\v a_j$ and standard deviations $\v \gs_j$ for $j = 1,\ldots, r$. 
We can then evaluate the cdf of each component at (say) $\v a_j \pm k \v \gs_j$ for prescribed values of $k$.
See a numerical illustration in \cref{sec:density}.

We also have significant theoretical contributions.
In \cref{thm:coupled}, we derive a \textup{novel system} of coupled CP decompositions, relating the moment tensors of $\mathcal{D}$ to the mixing weights and  moments of $\mathcal{D}_j$.
This has already seen interest in computational algebraic geometry \cite{alexandr2023moment}.
In \cref{thm:low_rank_unique} and \ref{thm:wMid}, we prove for the model \eqref{eq:cond-inp-mix} in $\mathbb{R}^n$ that if $d \geq 3$ and $r = \cO(n^{\floor{d/2}})$ then the mixing weights and means of $\mathcal{D}_j$ are \textup{identifiable} from the \textit{first $d$ moments of $\mathcal{D}$}.
In \cref{thm:gmeans_id}, we prove that the general means \eqref{eq:fnal} are identifiable from the \textit{weights and means} and \textit{certain expectations over $\mathcal{D}$}.
These guarantees are new, and more relevant for MoM than prior identifiability results which assumed access to the joint distribution  \eqref{eq:cond-inp-mix} \cite{allman_identifiability_2009}.
In \cref{thm:alsconv}, we prove a guarantee for our algorithm, namely local convergence of ALS (step 1).  

Our last contribution is extensive numerical experiments presented in \cref{sec:exp}.  
We test our procedures on various simulated estimation problems
as well as real data sets.  
The tests demonstrate the scalability of our methods to high dimensions $n$, and their competitiveness with leading alternatives.
The experiments also showcase the  applicability of the model \eqref{eq:cond-inp-mix} to different problems in data science.
A Python implementation of our algorithms is available at \url{https://github.com/yifan8/moment_estimation_incomplete_tensor_decomposition}.

\subsection{Notation} 
\label{sec:notation}

Vectors, matrices and tensors are denoted by lower-case, capital and calligraphic letters respectively, for example by $\v a, \m A, \t A$.
The 2-norm of a vector and the Frobenius norm of a matrix or tensor are denoted by $\| \cdot \|$.
The Euclidean inner product between vectors and the Frobenius inner product between matrices or tensors are denoted by $\langle \cdot \, , \cdot \rangle $.
Operators $\ast$ and $/$ denote entrywise multiplication (or exponentiation) and division respectively.
We write the tensor product as $\otimes$. 
We write $\mathbb{R}^n \otimes \ldots \otimes \mathbb{R}^n  = \mathbb{R}^{n^d}\!$.
The subspace of order $d$ symmetric tensors is denoted $S^{d}(\mathbb{R}^n)$.
For a matrix $\m A$, its $i$th column and $i$th row are  $\v a_i$ and $\v a^i$, respectively. 
The probability simplex is denoted as $\gD_{r-1}  = \{ \v x \in \R^r_{\geq 0}: \sum_{i} x_i = 1\}$.  
The all-ones vector or matrix is $\ind{}$.
An integer partition of $d \in \mathbb{N}$ is a tuple of non-increasing positive integers $\v \lambda = (\lambda_1, \lambda_2, \ldots )$ such that 
$\lambda_1 + \lambda_2 + \ldots = d$, denoted $\v \lambda \vdash d$. 
We write $\v \lambda = 1^{b_1} 2^{b_2} \ldots$ if $b_1$ elements are $1$, etc.
The length of $\v \gl$ is $\ell(\v \lambda) = b_1 + b_2 + \ldots$

    \section{Method of Moments and Identifiability}\label{sec:mom}

In this section we formulate the MoM framework for solving model \eqref{eq:cond-inp-mix}.  We derive moment equations that we will solve and present important theory on the uniqueness of the equations' solution.

\subsection{Setup}  Let $\mathcal{D}$ be a conditionally-independent mixture of $r$ distributions in $\mathbb{R}^n$  \eqref{eq:cond-inp-mix}.
We assume $r$ is known (or has been estimated).  For data we assume we are given a matrix of $p$ independent draws from $\mathcal{D}$:
  \begin{equation} \label{eq:data}
\m V \, = \, \left(\v v_1 \, | \, \ldots \, | \, \v v_p \right) \, \in \, \mathbb{R}^{n \times p}.
 \end{equation} 
For the nonparametric estimation of \eqref{eq:cond-inp-mix} by MoM, our goal is to evaluate the functional 
 \begin{equation} \label{eq:g-functional}
    g \mapsto \E_{X \sim \cD_j}[g(X)], \ \ \text{where}\ \ g(\v x) = (g_1(x_1),\ldots,g_n(x_n)),
\end{equation}  
which acts on functions $g$ that are Cartesian products of functions $g_i$.  We do this using appropriate expectations over $\mathcal{D}$ approximated by sample averages over $\m V$.

To this end we use the $d$th moment tensor of $\mathcal{D}$, defined as
\begin{equation} \label{eq:population-moment}
\t M^d  \, =  \, \E_{X \sim \mathcal{D}}[X^{\otimes d}]  \, \in \, S^d(\mathbb{R}^n).
\end{equation}
It is approximated by the $d$th
sample moment computed from the data matrix $\m V$:
\begin{equation} \label{eq:sample}
{\widehat{\t M}}^{\! d} \, = \, \frac{1}{p} \sum\nolimits_{i=1}^p \v v_i^{\otimes d} \, \in \, S^d(\mathbb{R}^n).
 \end{equation}
For functions $g : \mathbb{R}^n \rightarrow \mathbb{R}^n$ as in \eqref{eq:g-functional}, we will also use the expectations 
\begin{equation}\label{eq:my-weird-expectation}
    \mathbb{E}_{X \sim \cD} [g(X) \otimes X^{\otimes d-1}] \, \in \, \mathbb{R}^{n^d}.
\end{equation}
These are approximated by the sample averages
\begin{equation}\label{eq:my-moment}
    \frac{1}{p} \sum\nolimits_{i=1}^p g(\v v_i) \otimes \v v_i^{\otimes d-1} \, \in \, \mathbb{R}^{n^d}.
\end{equation}
 
In terms of the variables we wish to solve for, let the $d$th 
moment of component $\cD_j$ be 
\begin{equation} \label{eq:componentwise-moments}
\v m_j^{d} \, = \,  \E_{X \sim \mathcal{D}_j}[X^{\ast d}] \, \in \, \mathbb{R}^n.
\end{equation}
For the ``general means" in \eqref{eq:g-functional}, let
\begin{equation}\label{eq:my-y-moment}
\v y_j = \E_{X \sim \cD_j} [g(X)]  \in \mathbb{R}^n 
\end{equation}
and $\m Y = (\v y_1 \, | \, \ldots \, | \, \v y_r) \in \R^{n\times r}$, when $g$ is fixed and integrable with respect to $\mathcal{D}_j$.
We want to solve for \eqref{eq:componentwise-moments} (including the componentwise means) from \eqref{eq:sample}, and \eqref{eq:my-y-moment} from \eqref{eq:my-moment}.

\begin{remark}\label{rem:scalability}
To achieve scalability of MoM to high dimensions, it is mandatory that we never explicitly form high-dimensional high-order tensors.  The costs of forming $n^d$ tensors are prohibitive.
We evade them by developing new {\emph{implicit tensor decomposition methods}}, motivated by \cite{sherman2020estimating,pereira2022tensor} and kernel methods.  See \cref{sec:alg} \nolinebreak and~\ref{sec:gmeans}.
\end{remark}

\subsection{Systems of Equations}
We derive the 
relationship between the moments $\t M^d$ and $\widehat{\t M}^d$ of the mixture $\mathcal{D}$ and the moments $\v m_j^d$ and general means $\v y_j$ of the components $\mathcal{D}_j$. 
The relationship is interesting algebraically \cite{alexandr2023moment}, and forms the basis of our algorithm development.

\begin{definition} \label{def:projector}
Let $\v \lambda \vdash d$ be an integer partition.  
Define a projection operator $\m P_{\v \lambda} : \mathbb{R}^{n^d} \rightarrow \mathbb{R}^{n^{\ell(\v \lambda)}}$ acting on order $d$ tensors and outputing order $\ell(\v \lambda)$ (see \cref{sec:notation} for the definition) tensors via
\begin{equation*}
 \m P_{\v \lambda}(\t A)_{i_1, \ldots, i_{\ell(\v \lambda)}} = \begin{cases} 0 & \text{if } i_1, \ldots, i_{\ell(\v \lambda)}  \text{are~not~distinct} \\
\t A_{i_1, \ldots, i_1 (\lambda_1 \textup{ times}), i_2, \ldots, i_2 (\lambda_2 \textup{ times}), \ldots} & \text{otherwise}. 
\end{cases}
\end{equation*}
We write $\m P \equiv \m P_{(1, \ldots, 1)}$ when $\v \lambda$ consists of all $1$\!'s.  
\end{definition}

In particular, $\m P$ acts  by off-diagonal projection, by zeroing out all entries with non-distinct indices.
\cref{fig:diagonals} illustrates the other projections $\m P_{\v \lambda}$ when $d=3$.

\begin{figure}[htbp]
    \centering
    \subfloat{\includegraphics[width=0.25\linewidth]{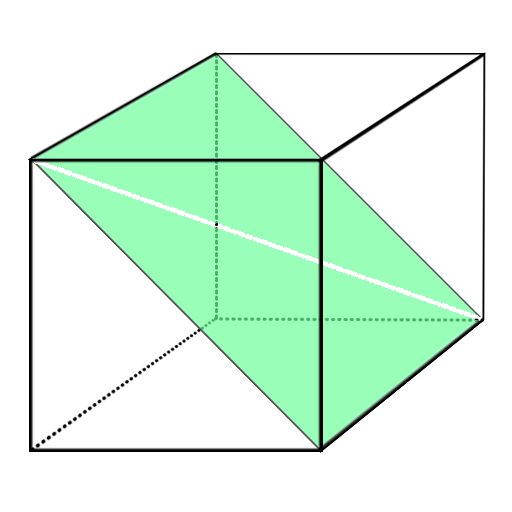}}\hspace{0.05\linewidth}
    \subfloat{\includegraphics[width=0.25\linewidth]{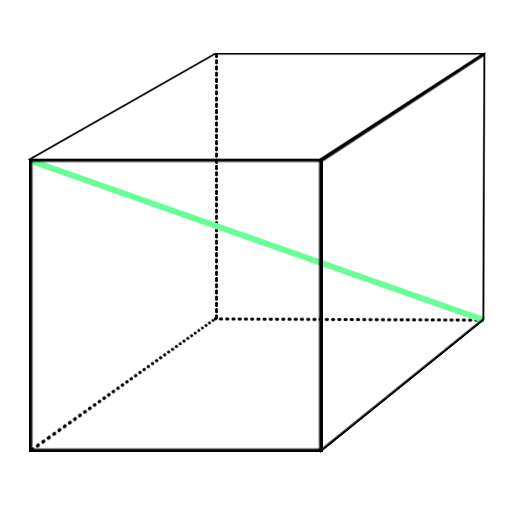}}
    \caption{Illustration of \cref{def:projector} when $d=3$.  Left: the entries selected by $\m P_{(2,1)}$ (green) with indices $iij$, $i \neq j$, resulting in a matrix indexed by $(i, j)$.   Right: the entries selected by $\m P_{(3)}$ with indices $iii$, resulting in a vector indexed by $i$.}
    \label{fig:diagonals}
\end{figure}

\begin{theorem} \label{thm:coupled} 
Let $\mathcal{D}$ be a conditionally-independent mixture, and $\t M^d$ its dth population moment tensor  \eqref{eq:population-moment}.
    Then for each  partition $\v \lambda = (\gl_1,\ldots,\gl_{\ell(\v \gl)})$ of $d$, 
    \begin{equation} \label{eq:desired}
        \m P_{\v \lambda}(\t M^d) = \m P \big{(} \sum\nolimits_{j=1}^r w_j \, \v m_j^{\lambda_1} \otimes \ldots \otimes \v m_j^{\lambda_{\ell(\v \lambda)}} \big{)}. 
    \end{equation}
\end{theorem}

\begin{proof}
    From \eqref{eq:cond-inp-mix},  \eqref{eq:population-moment} and conditioning on the class label $j$,
    \begin{equation} \label{eq:calc-1}
        \t M^d \, = \, \E_{X \sim \mathcal{D}}[X^{\otimes d}] \, = \, \sum\nolimits_{j=1}^r w_j  \E_{X \sim \mathcal{D}_j}[X^{\otimes d}].
    \end{equation}
    Fix the partition $\v \lambda \vdash d$ and write $\ell = \ell(\v \lambda)$.  
    We consider an index 
    \begin{equation*}
        \v i = (\underbrace{i_1, \ldots, i_1}_{\lambda_1 \textup{ times}}, \ldots, \underbrace{i_{\ell}, \ldots, i_{\ell}}_{\lambda_{\ell} \textup{ times}})
    \end{equation*}
    such that $i_1, i_2, \ldots, i_{\ell} \in [n]$ are distinct. 
    By coordinatewise independence in $\mathcal{D}_j$, 
    \begin{align*} 
        \E_{X \sim \mathcal{D}_j}[X^{\otimes d}]_{\v i} \, 
        &= 
        \, \E_{X \sim \mathcal{D}_j}[X_{i_1}^{\lambda_1} \ldots X_{i_{\ell}}^{\lambda_{\ell}}] \, = \, \E_{X \sim \mathcal{D}_j}[X_{i_1}^{\lambda_1}] \ldots \E_{X \sim \mathcal{D}_j}[X_{i_\ell}^{\lambda_{\ell}}] \\
        &= 
        \, (\v m_j^{\lambda_1})_{i_1} \ldots (\v m_j^{\lambda_{\ell}})_{i_{\ell}} = (\v m_j^{\lambda_1} \otimes \ldots \otimes \v m_j^{\lambda_{\ell}})_{i_1, i_2, \ldots, i_{\ell}}
    \end{align*}
    where we used the definition \eqref{eq:componentwise-moments}.  Ranging over such indices $\v i$ gives
    \begin{equation} \label{eq:calc-2}
        \m P_{\v \lambda}\big{(}  \E_{X \sim \mathcal{D}_j}[X^{\otimes d}] \big{)} \, = \, \m P( \v m_j^{\lambda_1} \otimes \ldots \otimes \v m_j^{\lambda_{\ell}}).
    \end{equation}
    Finally we put \eqref{eq:calc-1} together with \eqref{eq:calc-2} to deduce \eqref{eq:desired}.  
\end{proof}

Theorem~\ref{thm:coupled} generalizes results in \cite{guo2022learning}.  It relates the mixture moments to the componentwise moments via a coupled system of incomplete tensor decompositions.

\begin{example}\label{ex:coupled}
    We write out \eqref{eq:desired} in \cref{thm:coupled} for all moments of order $\leq 3$:
    \begin{align*}
    \t M^1   & = \, \sum\nolimits_{j=1}^r \! w_j \v m_j^1 \hspace{-6.5em} &\in \mathbb{R}^n \\[-0.1pt]
    \m P (\t M^2)  &= \, \m P \big{(} \sum\nolimits_{j=1}^r \! w_j  (\v m_j^1)^{\otimes 2}\big{)} \hspace{-6.5em}  &\in  S^2(\mathbb{R}^n)\\[-0.1em]
    \m P_{(2)}(\t M^2)  & = \, \sum\nolimits_{j=1}^r \! w_j \v m_j^2 \hspace{-6.5em} &\in \mathbb{R}^n \\[-0.1pt]
    \m P(\t M^3) &= \, \m P\big{(}\sum\nolimits_{j=1}^r \! w_j (\v m_j^1)^{\otimes 3}\big{)} \hspace{-6.5em} &\in S^3(\mathbb{R}^n) \\[-0.1pt]
    \m P_{(2,1)}(\t M^3) & = \, \m P\big{(} \sum\nolimits_{j=1}^r \! w_j \v m_j^2 \otimes \v m_j^1 \big{)} \hspace{-6.5em} &\in \mathbb{R}^{n \times n} \\[-0.1pt]
    \m P_{(3)}(\t M^3)  & = \, \sum\nolimits_{j=1}^{r} \! w_j \v m_j^3 \hspace{-6.5em} &\in \mathbb{R}^n.
    \end{align*}
    These equations are all of the information that the first three mixture moments $(\t M^d)$ contain regarding the componentwise moments  $(\v m_j^d)$.  
    Notice that the off-diagonal entries $\m P(\t M^d)$ belong to a diagonally-masked symmetric CP decomposition with components given by the mean vectors $\v m_j^1$.  The diagonals of $\t M^d$ involve higher-order moments $\v m_j^d$ as well.  The diagonals have partially symmetric CP decompositions.
\end{example}

Regarding the general means $\m Y$, the next proposition relates the joint expectation \eqref{eq:my-weird-expectation} to the componentwise general means \eqref{eq:my-y-moment}.  These equations are linear in $\m Y$.

\begin{proposition}\label{prop:general-equations}
Let $\mathcal{D}$ be a conditionally-independent mixture, and $\m Y$ its general means matrix  for a function $g$ \eqref{eq:my-y-moment}.  Then for all $d \geq 1$,
\begin{equation}\label{eq:general-means}
\m P \Big{(} \mathbb{E}_{X \sim \cD} [g(X) \otimes X^{\otimes d-1}]\Big{)} = \m P\Big{(} \sum\nolimits_{j=1}^r w_j \v y_j \otimes (\v m_j^1)^{\otimes d-1}\Big{)}.
\end{equation}
\end{proposition}

\begin{proof}
Fix an arbitrary index $(i_1,\ldots,i_d) \in [n]^d$ with distinct entries on the off-diagonal unmasked by $\P$. 
Let $X_{ij}$ be independent random variables with distributions $\cD_{ij}$. 
Then by conditional independence, we have
\begin{align*}
    \m P \Big{(} \mathbb{E}_{X \sim \cD} [g(X) \otimes X^{\otimes d-1}]\Big{)}_{i_1,\ldots,i_d}
    &= \,
    \sum\nolimits_{j=1}^r w_j
    \E g_{i_1}(X_{i_1 j})\prod\nolimits_{s = 2}^d \E X_{i_s j}\\
    &= \,
    \sum\nolimits_{j=1}^r w_j
    (\v y_j)_{i_1}\prod\nolimits_{s = 2}^d (\v m_j^1)_{i_s}\\
    &= \,
    \m P\Big{(} \sum\nolimits_{j=1}^r w_j \v y_j \otimes (\v m_j^1)^{\otimes d-1}\Big{)}_{i_1,\ldots,i_s}.
\end{align*}
Thus the two tensors in \eqref{eq:general-means} are equal as claimed.
\end{proof}

\subsection{Identifiability} \label{subsec:identify}
It is important to know that equations have a unique solution before developing algorithms to solve them.  
Here we prove that the system of incomplete tensor decompositions in 
\cref{thm:coupled} uniquely identify the weights and means.  
Also the linear systems in 
\cref{prop:general-equations} uniquely identify the general means.

We present a result in tensor completion, which may be of independent interest.

\begin{theorem} \label{thm:low_rank_unique}
    Let $2 < d < n$ and $r \leq \binom{\floor{(n-1)/2}}{\floor{d/2}}$.   Let   
    $\m A \in \mathbb{R}^{n \times r}$ be Zariski-generic\footnote{
        A property $\operatorname{P}(x)$ on $\mathbb{R}^N$ is said to hold for \textit{Zariski-generic} $x$ if there exists a nonzero polynomial $F$ on $\mathbb{R}^N$ such that $F(x) \neq 0$ implies $\operatorname{P}(x)$  \cite{harris2013algebraic}. 
        In particular,  if  $x$ is random and drawn from any absolutely continuous probability measure on $\mathbb{R}^N$, then $\operatorname{P}(x)$ holds with probability $1$.} matrix.
        If 
    \begin{equation}
    \m P \big{(}\sum\nolimits_{j=1}^r \v a_j^{\otimes d}\big{)} \, = \, \m P \big{(}\sum\nolimits_{j=1}^r \v b_j^{\otimes d}\big{)}
    \end{equation}
        for $\m B \in \mathbb{R}^{n \times r}$, then $\m A$ and $\m B$ are equal up to trivial ambiguities.  That is,
    $
    \m A \, = \, \m B \m \Pi
    $
    where $\m \Pi \in \mathbb{R}^{r \times r}$ is a permutation if $d$ is odd and a signed permutation  if $d$ is even. 
\end{theorem}

Our main tool to prove Theorem~\ref{thm:low_rank_unique} is the following lemma.

\begin{lemma} \label{lem:full_rk_flat_sym}
Let $\m A \in \mathbb{R}^{n \times r}$ and $\t T =  \sum\nolimits_{j=1}^r \v a_j^{\otimes s+t}  \in S^{s+t}(\mathbb{R}^n)$ with $s \leq t$. 
Let $\m T_0 \in \mathbb{R}^{\binom{n+s-1}{s} \times \binom{n+t-1}{t}}$ be the reduced matrix flattening of $\t T$ defined by
\begin{equation}\label{eq:my-flatten}
( \m T_0 )_{\v i, \v j} \, = \, \t T_{i_1, \ldots, i_s, j_1, \ldots, j_t}
\end{equation}  
for all non-decreasing words  $\v i$ and  $\v j$ over  $[n]$ of lengths  $s$ and  $t$ respectively.
If $\m A$ is Zariski-generic, then each $m \times m$ submatrix of $\m T_0$ has rank $\min \! \big{\{} m, r \big{\}}$.
\end{lemma}

Complete proofs are given in \cref{apx:pf_low_rank_unique}, but a sketch is included below.  

\begin{proof}[Proof Sketch (\cref{thm:low_rank_unique})] We show  the masked tensor $\t T = \m P \big{(} \sum\nolimits_{j=1}^r \v a_j^{\otimes d} \big{)}$ can be completed uniquely to a tensor of CP rank $r$. 
This is by considering the matrix flattening $\m T_0$ in \eqref{eq:my-flatten}. 
Locate $(r+1) \times (r+1)$-sized submatrices of $\m T_0$ where only one entry is missing, and use  \cref{lem:full_rk_flat_sym} to 
uniquely fill in the entry. 
The masked diagonal entries in $\t T$ are recovered in this way, sequentially according to the repetition pattern of its index, $\v \gl \vdash d$, in the following order:
\begin{align}\label{eq:fillpath}
    &2^11^{d-2} \, \rightarrow \, 2^21^{d-4} \, \rightarrow \, \ldots \, \rightarrow \,
    2^{\floor{d/2}}1^{d-2\floor{d/2}} \, \rightarrow \\
    &a_1^{b_1}\ldots a_s^{b_s}~ (a_i \leq 4) \, \rightarrow \,
    a_1^{b_1}\ldots a_s^{b_s}~ (a_i \leq 8) \, \rightarrow\ldots  \nonumber 
\end{align}   
After completing $\t T$, the rest follows from identifiability of CP~decompositions \nolinebreak \cite{chiantini2017generic}.  
\end{proof}

Theorem~\ref{thm:low_rank_unique} implies identifiability of the component weights and means from moments of $\cD$. 
    To ease the notation, from now on we denote the mean matrix as
    \begin{equation*}
        \m A = (\v a_1|\ldots|\v a_r) := (\v m_1^1|\ldots|\v m_r^1) \in \R^{n\times r}.
    \end{equation*}

\begin{theorem} \label{thm:wMid}
    Let $\mathcal{D}$ be a conditionally-independent mixture \eqref{eq:cond-inp-mix} with weights $\v w \in \Delta_{r-1}$  such that $w_j > 0$ for each $j$ and Zariski-generic means $\m A \in \mathbb{R}^{n \times r}$.
    Let $d_1, d_2 \in \mathbb{N}$ be distinct  such that $d_1 \geq 3$, $r \leq \min\set{\binom{\floor{(n-1)/2}}{\floor{d_1/2}}, \binom{n}{d_2}}$.
    Then $\v w$ and $\m A$ are uniquely determined from the equations
        \begin{equation}
            \P\big{(}\sum\nolimits_{j=1}^r w_j \v a_j^{\otimes d_1}\big{)} = \P\big{(}\t M^{d_1}\big{)},\ \ \ \ \ 
            \P\big{(}\sum\nolimits_{j=1}^r w_j \v a_j^{\otimes d_2}\big{)} = \P\big{(}\t M^{d_2}\big{)},
        \end{equation}
    up to possible sign flips on each $\v a_j$ if $d_1$ and $d_2$ are both even.
\end{theorem}

\begin{proof}
    We absorb the weights into vectors by defining $\v u_j^{\otimes d_1} = w_j \v a_j^{\otimes d_1}$.
    If $d_1$ is odd, by \cref{thm:low_rank_unique} $\v u_j = w_j^{1/d_1} \v a_j$ are uniquely determined by $\m P(\t M^{d_1})$.
    Then
    \begin{equation} \label{eq:wlineq}
        \P\big{(}\sum\nolimits_{j = 1}^r w_j {\v a_j}^{\otimes d_2}\big{)}
       \, = \,
        \sum\nolimits_{j = 1}^r w_j^{1 - \frac{d_2}{d_1}} \P\big{(}{\v u_j}^{\otimes d_2}\big{)}
        \, = \,
        \P\big{(}\t M^{d_2}\big{)}
    \end{equation}
    is a linear system in $\eprod{\v w}{(1 - \frac{d_2}{d_1})}$ that admits the true weights as a solution.
    Since $\m A$ is Zariski-generic and $\v w$ is positive, 
    $\{ \P\left({\v u_j}^{\otimes d_2}\right)\}_{j=1}^r$ is linearly independent (see \cref{lem:offdiag_indep}).
    Hence \eqref{eq:wlineq} uniquely determines $\v w$.
    Uniqueness of $\v a_j = w_j^{-1/d} \v u_j$ also follows.
    The cases where $d_1$ is even are analogous. 
\end{proof}

\begin{remark}
    In \cite{chen2019beyond}, the authors proved that the first $r$ moments are sufficient to identify the means and weights for conditionally-independent mixtures with positive weights. 
    \cref{thm:wMid} shows that in fact for almost all mean matrices, only two of the first 
    \begin{equation*}
        d^* 
        = \min \set{d \in \mathbb{N}: r \leq 
            {\scriptsize
                \begin{pmatrix}
                    \floor{(n-1)/2}\\
                    \floor{d/2}
                \end{pmatrix}
            }
        } 
        = 
        \cO(\log r / \log n)
    \end{equation*}
    moments are required. 
    In particular, for $r < n/2$, the first 3 moments are sufficient. 
\end{remark}

Another important result is the identifiability of the general means \eqref{eq:my-y-moment} from the weights, means and the left-hand side of \eqref{eq:general-means}, as stated below.

\begin{theorem}\label{thm:gmeans_id}
    Let $\mathcal{D}$ be a conditionally-independent mixture with Zariski-generic means $\m A$ and strictly positive weights $\v w$.  
    Then $\v y_j = \E_{X \sim \cD_j}[g(X)]$ are identifiable from $(\v w, \m A)$ and the off-diagonal entries $\P\big{(}\E_{X\sim \cD}(g(X) \otimes X^{\otimes d-1})\big{)}$
    if $r\leq \binom{n-1}{d-1}$.
\end{theorem}

We give the proof in \cref{apx:pf_gmeansid}.  In particular, \cref{thm:gmeans_id} implies that if $(\v w, \m A)$ are identifiable from the first $d$ moments of $\cD$ (e.g., under the setting of \cref{thm:wMid}), then the first $d'$ componentwise moments \eqref{eq:componentwise-moments} are identifiable from the first $(d' + d - 1)$ moments of $\cD$.

    \section{The ALS Algorithm for Means and Weights} \label{sec:alg}
In this section we present an ALS algorithm to solve for the weights and means of the mixture given data samples.  The general means solver is presented in \cref{sec:gmeans}.
As suggested by \cref{thm:wMid}, we choose at least two different moments $d_1$ and $d_2$, and find weights $\v w$ and means $\m A$ that best approximate the joint moments.
Thus, we use a cost function 

\begin{align} 
    & \min_{\substack{\v w \in \Delta_{r-1} \\ \m A \in \mathbb{R}^{n \times r}}} \, f^{[d]} := \gt_1 f^{(1)} + \ldots + \gt_d f^{(d)} \label{eq:costfndef} \\[0.5pt] 
    & \quad  \quad \text{where } f^{(i)}(\v w, \m A; \m V) 
    := \big{\|}\P \big{(}{\widehat{\t M}}^{\! i} - 
    \sum\nolimits_{j=1}^r \! w_j \v a_j^{\otimes i}\big{)}
    \big{\|}^2. \nonumber
\end{align}
Here $\m V \in \R^{n\times p}$ is the given matrix of data observations \eqref{eq:data}.  
The scalars $\gt_i \geq 0$ are hyperparameters.
The reason for choosing the $\ell^2$ cost is that it enables efficient ALS optimization routines, in which there is no need to explicitly compute $\widehat{\t M}^i$ and ${\v a_j}^{\otimes i}$ owing to implicit or kernel based methods (cf. \cref{rem:scalability}).
We choose ALS in view of the multilinearity of the cost function (explained below).
Alternatively, first-order optimization methods like  gradient descent or quasi-Newton procedures (e.g., BFGS) could be used.  However, ALS also ran faster in our experiments.

\subsection{Overview} \label{sec:multilinear}
We outline the main points of the algorithm.
Finding $(\v w, \m A)$ that minimizes $f^{[d]}$ is an incomplete (due to the off-diagonal mask $\m P$) symmetric CP tensor decomposition problem, see \cref{ex:coupled}. 

Notice that in the cost function \eqref{eq:costfndef}, the off-diagonal mask $\m P$ creates a useful multilinearity.  
Namely the residuals:
\begin{equation} \label{eq:residual}
    \P \big{(}{\widehat{\t M}}^{\! i} - 
    \sum\nolimits_{j=1}^r \! w_j \v a_j^{\otimes i}\big{)}
    =
    \P\big(\,\frac{1}{p}\sum_{\ell = 1}^p\tprod{\v v_\ell}{i} - 
    \sum\nolimits_{j=1}^r \! w_j \v a_j^{\otimes i}\big)
 \end{equation}
are separately linear in the weights $\v w$ and each {\emph{row}} of the mean matrix $\m A$ (but {\emph{not}} in the columns of $\m A$).  
Multilinearity suggests an alternating least squares (ALS) scheme for the minimization of \eqref{eq:costfndef}.
That is, we alternatingly update the weight vector or one row of the mean matrix, while keeping all other variables fixed.

Each subproblem in this ALS scheme is a linear least squares problem, where the variable lies in just $\mathbb{R}^r$.  However there are $\mathcal{O}(n^d)$ linear equations (coming from the entries of \eqref{eq:residual}).
Thus in high dimensions, it is exorbitant to form the coefficient matrix directly.  
One of our key insights is that it is possible to compute the $r \times r$ normal equations for these least squares, while avoiding $\mathcal{O}(n^d)$-sized coefficient matrices.

Specifically we show that the entries in the $r\times r$ matrix for the normal equations are given by \textit{kernel evaluations} $K_i(\v a_j, \v a_{j'})$ and $K_i(\v a_j, \v v_{\ell})$, where
\begin{equation}\label{eq:kernel}
    K_i(\v x, \v y) := \ip{\P\tprod{\v x}{i}}{\P\tprod{\v y}{i}}.
\end{equation}
We prove that these kernels can be efficiently evaluated from the Gram matrices
\begin{equation*}
    \m G^{A, A}_s \! = \left(\eprod{\m A}{s}\right)^{\!\top} \! \left(\eprod{\m A}{s}\right)  \in \mathbb{R}^{r \times r},   \quad   
    \m G^{A, V}_s \! = \left(\eprod{\m A}{s}\right)^{\! \top} \! \left(\eprod{\m V}{s}\right)  \in \mathbb{R}^{r \times p} \quad (s = 1, \ldots, d).
\end{equation*}
In this way we obtain a tensor-free ALS algorithm, summarized in \cref{alg:basicALS}.

\begin{algorithm}
    \begin{algorithmic}[1]
        \caption{Baseline ALS algorithm for solving means and weights}
            \label{alg:basicALS}
            \Function{SolveMeanAndWeight}{data $\m V$, initialization ($\v w$, $\m A$), order $d$, hyperparameters $\v \gt$}
            \State Compute $\{\m G^{A, A}_s\}_{s = 1}^d$ and $\{\m G^{A, V}_s\}_{s = 1}^d$
            \While{not converged}
                \State 
                $\m A, \{\m G^{A, A}_s\}_{s = 1}^d, \{\m G^{A, V}_s\}_{s = 1}^d 
                \! \gets 
                \small{\mstep(\{\m G^{A, A}_s\}_{s = 1}^d, \{\m G^{A, V}_s\}_{s = 1}^d, \v w, \m A, d, \v \gt)}$
                \State 
                $\v w 
                \gets 
                \wstep(\{\m G^{A, A}_s\}_{s = 1}^d, \{\m G^{A, V}_s\}_{s = 1}^d, d, \v \gt)$
            \EndWhile
            \EndFunction
        \Return $\v w$, $\m A$
    \end{algorithmic}
\end{algorithm}

Here $\mstep$ (\cref{alg:Mstep}) updates the mean matrix row by row as well as the $\m G$ matrices, and $\wstep$ (\cref{alg:wstep}) updates the weights.
We present the details of $\wstep$ in \cref{sec:wstep} and $\mstep$ in \cref{sec:mstep}.
Analyses therein readily give the complexity of our baseline ALS algorithm:

\begin{lemma}
    Each loop of \textbf{while} in \cref{alg:basicALS} takes $\cO(npr + nr^3 + T_{QP}(r))$ flops and $\cO(r(n+p))$ storage, assuming $d = \cO(1)$. Here $T_{QP}(r)$ is the flop count for solving the convex quadratic programming problem \eqref{eq:qp} in $\R^r$. 
\end{lemma}

We stress that costs of $n^d$ are evaded due to our use of the kernel \eqref{eq:kernel}, as detailed next.
This is a main reason why the ALS framework is well-suited for high dimensions.  
Further improvements on the baseline algorithm are recommended in \cref{sec:improve}, yielding our final algorithm for computing means and weights in \cref{alg:finalALS}.

\subsection{Fast Update of Weights} \label{sec:wstep}
We explain $\wstep$ in \cref{alg:basicALS}.
Consider a generalized weight update problem, where $\v \pi \in \mathbb{R}^{p}$ is a scaling vector:
\begin{align}
&  \min_{\v w \in \mathbb{R}^r} f^{[d]}(\v w; \m A, \v \pi, \m V) := \sum\nolimits_{i = 1}^d \gt_i f^{(i)}(\v w; \m A, \v \pi, \m V) \label{eq:costfn_gen} \\[4.5pt] 
&  \quad  \quad \text{where } f^{(i)}(\v w; \m A, \v \pi, \m V) 
:=  
\big{\|} \P \big{(}\sum\nolimits_{\ell=1}^p \pi_{\ell} \v v_{\ell}^{\otimes i} - \sum\nolimits_{j=1}^r w_j \v a_j^{\otimes i}\big{)} \big{\|}^2. \nonumber
\end{align} 
If we choose $\pi_\ell = 1/p$ and impose the simplex constraint $\v w \in \Delta_{r-1}$, then we recover the weight update problem in ALS for \eqref{eq:costfndef}.
The goal is to form the $r\times r$ normal equation for $\v w$ without ever forming any tensors. As mentioned we use the kernels
\begin{equation*}
    K_i(\v x, \v y) = \ip{\P\tprod{\v x}{i}}{\P\tprod{\v y}{i}}.
\end{equation*}
Expanding the square in \eqref{eq:costfn_gen}, the cost function can be written as 
\begin{equation} \label{eq:implcost_esp}
    f^{[d]}(\v w; \m A, \v \pi, \m V) 
    = 
    \sum_{i = 1}^d \gt_i \big{(}
    \sum_{j, \ell} w_j\pi_\ell \cdot K_i(\v a_j, \v v_\ell)
    +
    \sum_{j, j'} w_j w_{j'} \cdot K_i(\v a_j, \v a_{j'})
    \big{)}
     + C,
\end{equation}
where $C$ is a constant independent of $\v w$ and $\m A$.
Hence, we need to evaluate the kernels efficiently.
We make a connection with symmetric function theory in algebra.

\begin{definition}
    The elementary symmetric polynomial (ESP) of degree $d$ in $n$ variables \nolinebreak is 
    \begin{equation*}
        e_d(x_1,\ldots,x_n) :=  \sum\nolimits_{1 \leq i_1 <\cdots<i_d \leq n} x_{i_1}x_{i_2}\cdots x_{i_d}.
    \end{equation*}
    For $d=0$ define $e_0(\v x) := 1$.
\end{definition}

It is not hard to check that 
\begin{equation*}
    K_i(\v x, \v y) = i! e_i(\v x * \v y),
\end{equation*} 
where $*$ denotes elementwise multiplication.
To evaluate $K_i$ efficiently, we apply the Newton-Gerard formula \cite{macdonald1998symmetric} expressing ESPs in terms of power sum polynomials:
\begin{equation} \label{eq:esp2power}
     e_i(\v x) = \frac{1}{i} \sum\nolimits_{s = 1}^{i} \! e_{i-s}(\v x) (-1)^{s-1} p_s(\v x),
\end{equation}
where $p_s(\v x) := \sum_{j=1}^n x_j^s$.
This enables a recursive evaluation of kernels from the power sums $p_s(\v a_j * \v a_{j'})$ and $p_s(\v a_j * \v v_\ell)$.  These are the entries of the Gram matrices
\begin{equation} \label{eq:defG}
    \m G^{A, A}_s \! = \left(\eprod{\m A}{s}\right)^{\!\top} \! \left(\eprod{\m A}{s}\right)  \in \mathbb{R}^{r \times r},   \quad   
    \m G^{A, V}_s \! = \left(\eprod{\m A}{s}\right)^{\! \top} \! \left(\eprod{\m V}{s}\right)  \in \mathbb{R}^{r \times p} \quad (s = 1, \ldots, d).
\end{equation}
Altogether, \cref{alg:prepnormeqn} is how we compute the normal equation for $\v w$ in \eqref{eq:implcost_esp}.

\begin{algorithm}
    \begin{algorithmic}[1]
        \caption{Prepare normal equation $\m L \v w = \v r$ for weight update}
        \Function{PrepNormEqn}{$\{\m G_s^{A, A}\}_{s = 1}^d$, $\{\m G^{A, V}_s\}_{s = 1}^d$, scaling vector $\v \pi$, order $d$, hyperparameters $\v \gt$}
            \label{alg:prepnormeqn}
            \State
                $\m E^{A, A}_0 \gets \ind{r \times r}$, \,\, $\m E^{A, V}_0 \gets \ind{r \times p}$ 
            \For{$i = 1, ..., d$}
                \State 
                    $\m E^{A, A}_i \gets \frac{1}{i}\sum_{s = 1}^{i} \m E^{A, A}_{i-s} * (-1)^{s-1} \m G^{A, A}_s$
                \State 
                    $\m E^{A, V}_i \gets \frac{1}{i}\sum_{s = 1}^{i} \m E^{A, V}_{i-s} * (-1)^{s-1} \m G^{A, V}_s$
            \EndFor
            \State
                $\m L \gets \sum_{i = 1}^d i!\gt_i \m E^{A, A}_i$
            \State
                $\v r \gets \sum_{i = 1}^d i!\gt_i \m E^{A, V}_i \! \v \pi$ \vspace{0.5ex} 
            \EndFunction 
        \Return $\m L$, $\v r$
    \end{algorithmic}
\end{algorithm}

\begin{lemma}\label{prop:prepnormeqn_cost}
    Given the matrices $\m G^{A, A}_s$ and $\m G^{A, V}_s$\!,  \cref{alg:prepnormeqn} computes $(\m L, \v r)$ such that the normal equations for the weight update \eqref{eq:costfn_gen} are  $\m L \v w = \v r$ in $\cO(pr)$ flops and $\cO(r(n + p))$ storage, assuming $d = \cO(1)$.
\end{lemma}

\begin{proof}
    \!Recursively evaluating all ESPs by the Newton-Gerard formula takes $\cO(d^2)$  additions and entrywise multiplications of $r \times p$ and $r \times r$ matrices.
    So if $d = \mathcal{O}(1)$, \cref{alg:prepnormeqn} takes $\cO(r(r + p)) = \cO(pr)$ flops and $\cO(r(n + p))$ storage.
\end{proof}

\cref{alg:prepnormeqn} is the key step behind our implicit MoM method for estimating the weights and means.   Further, the normal equations generically have a unique solution.
\begin{proposition} \label{prop:invertL_w}
    Assume $\gt_d > 0$ and $\m A$ is Zariski-generic. Then the matrix $\m L \in \mathbb{R}^{r \times r}$ for the normal equations of \eqref{eq:costfn_gen} is invertible, provided $r \leq \binom{n}{d}$.
\end{proposition}

The proof is given in \cref{apx:pf_conv}.
In the weight update problem for our model, we have the simplex constraint as well.
This is not hard to handle in a tensor-free way.
Recall that for any (full rank) least squares problem constrained to set $S$ it holds
\begin{equation*}
\argmin_{\v x \in S} \| \m C \v x - \v b \|^2 \,\,\, = \,\,\, \argmin_{\v x \in S} \| \v x - \v x^{*} \|^2_{\m C^{\top} \m C},
\end{equation*}
where $\v x^{*} = \argmin \| \m C \v x - \v b \|^2$ is the unconstrained minimum and $\norm{\v x}_{\m C^{\top} \m C} := \| \m C \v x \|$.
Thus we let $\v w_0 = \m L^{-1} \v r$ be the unconstrained solution, and update $\v w$ according to
\begin{equation} \label{eq:qp}
\v w \, \leftarrow \, \argmin_{\v w \in \Delta_{r-1}} \| \v w - \v w_0 \|_{\m L}^2.
\end{equation}
Equation~\eqref{eq:qp} is a convex quadratic program (QP) that can be solved efficiently by standard algorithms.
The weight update algorithm is summarized in \cref{alg:wstep}.

\begin{algorithm}
    \begin{algorithmic}[1]
        \caption{Update the weights $\v w$}
            \label{alg:wstep}
            \Function{UpdateWeight}{$\{\m G^{A, A}_s\}_{s = 1}^d$, \!\!\! $\{\m G^{A, V}_s\}_{s = 1}^d$,  order $d$, hyperparameters \nolinebreak $\v \gt$}
            \State $\m L, \v r \gets \prepnormeqn(\{\m G^{A, A}_s\}_{s = 1}^d,~\{\m G^{A, V}_s\}_{s = 1}^d,~1/p \cdot \ind{},~d,~\v \gt)$
            \State $\v w_0 \gets \m L^{-1} \v r$
            \State Solve convex QP: $\v w \gets \argmin_{\v w\in \gD_{r-1}} \norm{\v w - \v w_0}_{\m L}^2$
            \EndFunction
        \Return $\v w$
    \end{algorithmic}
\end{algorithm}

\begin{remark}
    The cost function $f^{[d]}$ can be efficiently evaluated and monitored during the ALS iterations.
    Let $\m L$, $\v r$ be the output of $\prepnormeqn$ in the weight step and $\v w$ be the current weight. 
    The current cost is $\norm{\m L \v w - \v r}^2$.
\end{remark}

\subsection{Fast Update of Rows of Mean Matrix}\label{sec:mstep}
We explain $\mstep$ in \cref{alg:basicALS}.
Consider updating the $k$th row $\v a^k \in \mathbb{R}^r$ while keeping the weights and other rows in $\m A$ fixed.
We see from \cref{sec:multilinear} this is a linear least squares problem.
The goal is then to identify and compute its normal equation efficiently as in the weight update.
We show that it reduces to a general weight update problem \eqref{eq:costfn_gen}.
\begin{proposition} \label{prop:row_update}
    Let $\m A^{(k)}$ and $\m V^{(k)}$ be $\m A$ and $\m V$ with the $k$th row removed.
    Define $\m G^{A^{(k)}\!, A^{(k)}}_s \in \mathbb{R}^{r \times r}$ and $\m G^{A^{(k)},  V^{(k)}}_s \in \mathbb{R}^{r \times p}$ analogously to \eqref{eq:defG}.
    Let 
    \begin{equation*}
        \hat{\m L}, \hat{\v r} \leftarrow \prepnormeqn(\{\m G^{A^{(k)},  A^{(k)}}_s\}_{s = 1}^d,~\{\m G^{ A^{(k)},  V^{(k)}}_s\}_{s = 1}^d,~\frac{\v v^k}{p},~d-1,~\!(2\gt_2,\ldots,d\gt_d)),
    \end{equation*}
    and $\m L = \hat{\m L} + \gt_1$, $\v r = \hat{\v r} + \gt_1 \mu_k$, where $\v \mu = (1/p) \m V \ind{}$ is the sample mean.
    Then the normal equations for $\v \gb^k = \v w * \v a^k \in \mathbb{R}^{r}$ in \eqref{eq:costfndef} are given by $\m L \v \gb^k = \v r$.
\end{proposition}

\begin{proof}
    We start by identifying terms relevant to $\v a^k$ in the cost $f^{[d]}$ \eqref{eq:costfndef}.
    Define $\P^{(k)}$ as the projection sending order-$d$ tensors to order-$d$ tensors by
    \begin{equation} \label{eq:proj_k}
        \P^{(k)}(\t A)_{i_1\ldots i_{d}} :=  \begin{cases} 
        \t A_{i_1, \ldots, i_d} & \text{if } k, i_1, \ldots, i_d \text{ are distinct} \\ 
        0 & \text{otherwise}.
        \end{cases}
    \end{equation}
    We separate the least squares cost $f^{(i)}$ at order $i \geq 2$ in \eqref{eq:costfndef} as follows:
    \begin{align} \label{eq:separation}
        \Big{\|}\P\big{(}
            \sum\nolimits_{j = 1}^r \! w_j  \v a_j^{\otimes i}  - {\widehat{\t M}}^{\! i} 
        \big{)}\Big{\|}^2
        =
        &\underbrace{
            \Big{\|}\P^{(k)}\big{(}
                \sum\nolimits_{j = 1}^r w_j \v a_j^{\otimes i} - {\widehat{\t M}}^{\! i} 
            \big{)} \Big{\|}^2
        }_{\text{independent of  }\v a^k}  \nonumber\\
        &+
        i \, \underbrace{
            \Big{\|}
            \sum_{j = 1}^r w_j (\v a_j)_k \P^{(k)}\v a_j^{\otimes i-1} 
            - 
            \frac{1}{p}\sum_{\ell = 1}^p (\v v_\ell)_{k}  \P^{(k)}\v v_\ell^{\otimes i-1}
            \Big{\|}^2
        }_{\text{dependent on  }\v a^k} 
    \end{align}

    Writing
    $\v \gb^k = \v w * \v a^k \in \mathbb{R}^r$, 
    it reduces to finding $\v \gb^k$ that minimizes
    \begin{equation} \label{eq:alsd-1}
        i \, \Big{\|}\P \Big(\sum\nolimits_{j = 1}^r (\v \gb^k)_j (\v a_j^{(k)})^{\otimes i-1} - \frac{1}{p}\sum_{\ell = 1}^p (\v v_\ell)_{k}  (\v v^{(k)}_\ell)^{\otimes i-1} \Big{)}\Big{\|}^2,
    \end{equation}
    where for a vector $\v u$ we denote $\v u^{(k)}$ the vector with its $k$th entry removed.
    It follows that minimizing $\sum_{i \geq 2}\gt_i f^{(i)}$ with respect to $\v a^k$ is equivalent to solving
    \begin{equation}\label{eq:mybeta}
        \min_{\v \gb^k \in \mathbb{R}^r}
        \sum_{i = 2}^d \gt_i \cdot i\,
        \Big{\|} 
            \P \Big{(}
                \sum\nolimits_{\ell=1}^p \frac{(\v v^k)_{\ell}}{p} (\v v^{(k)}_\ell)^{\otimes i-1} 
                - 
                \sum\nolimits_{j=1}^r (\v \gb^k)_j \big(\v a^{(k)}_j\big)^{\otimes i-1}
            \Big{)} 
        \Big{\|}^2.
    \end{equation}
    Notice that this is an instance of the generalized weight update \eqref{eq:costfn_gen} in \cref{sec:wstep} with tensor order up to $d-1$ and vector length $n-1$.
    Specifically the normal equations for \eqref{eq:mybeta} are given by
    \begin{equation*}
        \hat{\m L}, \hat{\v r} \leftarrow \prepnormeqn(\{\m G^{A^{(k)},  A^{(k)}}_s\}_{s = 1}^d,~\{\m G^{ A^{(k)},  V^{(k)}}_s\}_{s = 1}^d,~\frac{\v v^k}{p},~d-1,~\!(2\gt_2,\ldots,d\gt_d)).
    \end{equation*}

    To incorporate the cost $\gt_1f^{(1)}$ at order $i = 1$, note the term dependent on $\v a^k$ is 
    $\gt_1|\langle \ind{}, \v \gb^k \rangle - \mu_k|^2$, where $\v \mu$ is the sample mean.
    To include this in the normal equations, adjust the output of \textsc{PrepNormEqn} by 
    $\m L = \hat{\m L} + \gt_1$ and $\v r = \hat{\v r} + \gt_1 \mu_k$.
\end{proof}

Crucially then, for the row updates of $\m A$ in ALS it also possible to evade costs of $\mathcal{O}(n^d)$ by never explicitly forming tensors.
The routine for obtaining and solving the normal equations for $\v \gb^k$ is summarized in \cref{alg:Mstep} below.
Note that once $\m G^{A\!, A}_s$ and $\m G^{A,  V}_s$ are computed, we can obtain $\m G^{A^{(k)}\!, A^{(k)}}_s$ and $\m G^{A^{(k)},  V^{(k)}}_s$ as rank-1 updates:
\begin{equation*}
    \m G^{A^{(k)}\!, A^{(k)}}_s = \m G^{A, A}_s - (\v a^k)^{*s} \left((\v a^k)^{*s}\right)^\top, \quad
    \m G^{A^{(k)},  V^{(k)}}_s = \m G^{A, V}_s - (\v a^k)^{*s} \left((\v v^k)^{*s}\right)^\top\!.
\end{equation*}

\begin{algorithm}
    \begin{algorithmic}[1]
        \caption{Update the mean matrix $\m A$}
            \label{alg:Mstep}
            \Function{UpdateMean}{data $\v V$, $\{\m G^{A, A}_s\}_{s = 1}^d$, $\{\m G^{A, V}_s\}_{s = 1}^d$, weights $\v w$, current means $\m A$, order $d$, hyperparameters $\v \gt$}
                \State $\v \mu \gets \frac{1}{p}\sum_{\ell = 1}^p \v v_\ell$
                \For{$k = 1, \ldots, n$}
                \Comment{\textbf{Solve the $k$th row}}
                    \For{$s = 1, ..., d$}
                        \Comment \textbf{rank-1 deflation}
                        \State 
                        $\m G^{A, A}_s \gets \m G^{A, A}_s - (\v a^k)^{*s} \left((\v a^k)^{*s}\right)^\top$
                        \State 
                        $\m G^{A, V}_s \gets \m G^{A, V}_s - (\v a^k)^{*s} \left((\v v^k)^{*s}\right)^\top$
                    \EndFor
                    \State $\v \pi \gets \v v^k/p$
                    \State 
                    $\v a^k\gets \textsc{SolveRow}(
                        k,
                        \{\m G^{A, A}_s\}_{s = 1}^d, 
                        \{\m G^{A, V}_s\}_{s = 1}^d,
                        \v w, \v \pi, \v \mu, d, \v \gt
                    )$
                    \State Update mean matrix $\m A(k,:) \gets \v a^k$
                    \For{$s = 1, ..., d$}
                        \Comment \textbf{update $\m G$ matrix}
                        \State 
                        $\m G^{A, A}_s \gets \m G^{A, A}_s + (\v a^k)^{*s} \left((\v a^k)^{*s}\right)^\top$
                        \State 
                        $\m G^{A, V}_s \gets \m G^{A, V}_s + (\v a^k)^{*s} \left((\v v^k)^{*s}\right)^\top$
                    \EndFor
                \EndFor
                \vspace{0.5ex}
            \EndFunction
            \Return $\m A, \{\m G^{A, A}_s\}_{s = 1}^d, \{\m G^{A, V}_s\}_{s = 1}^d$

            \vspace{4ex}

            \Function{SolveRow}{row index $k$, $\{\m G^{A, A}_s\}_{s = 1}^d$, $\{\m G^{A, V}_s\}_{s = 1}^d$, weights $\v w$, scaling vector $\v \pi$, data mean $\v\mu$, order $d$, hyperparameters $\v \gt$}
                
                \State $\m L, \v r \gets {\prepnormeqn(\{\m G^{A, A}_s\}_{s},\{\m G^{A, V}_s\}_{s},\v \pi,d-1,(2\gt_2 \ldots, d\gt_d))}$
                \State $\m L \gets \m L + \gt_1$ 
                \State $\v r \gets \v r + \gt_1  \mu_k$
                \State Solve for new row $\v a^k \gets (\m L^{-1} \v r) / \v w$
            \EndFunction
            \Return $\v a^k$
    \end{algorithmic}
\end{algorithm}

\begin{lemma}
    Given $\m G^{A, A}_s$ and $\m G^{A, V}_s$, \cref{alg:Mstep} updates the mean matrix $\m A$ and matrices $\m G^{A, A}_s$, $\m G^{A, V}_s$ in $\cO(npr + nr^3)$ flops and $\cO(r(n + p))$ storage, assuming $d = \cO(1)$.
\end{lemma}

\begin{proof}
    In each instance of the outer \textbf{\textit{for}} loop, the rank-1 updates take 
    $\cO(pr)$ flops.
    The normal equations are computed in $\cO(pr)$ flops by \cref{prop:prepnormeqn_cost}.
    Solving them takes $\cO(r^3)$ flops.
    Since there are $n$ calls to \textsc{SolveRow}, the total flop count is $\cO(npr + nr^3)$.
    Storage is comparable to updating the weights, and so $\cO(r(n + p))$.
\end{proof}

Furthermore we show the normal equations are generically full rank, and thus there is a unique update for $\v a^k$ in ALS.
\begin{proposition} \label{prop:invertL_a}
    Assume $\gt_d > 0$ and $\m A^{(k)}$ is Zariski-generic, where $\m A^{(k)}$ denotes $\m A$ with the $k$th row removed.  Then the matrix $\m L \in \mathbb{R}^{r \times r}$ for the normal equations in \cref{prop:row_update} is invertible, provided $r \leq \binom{n-1}{d-1}$.
\end{proposition} 

The proof is similar to that of \cref{prop:invertL_w}, and can be found in \cref{apx:pf_conv}.
By \cref{prop:invertL_a}, there is a unique optimal update for the row $\v a^k \gets (\m L^{-1} \v r) / \v w$ if the weights are strictly positive.

\subsection{Convergence Theory}\label{sec:alsconv}
Here we present results on the convergence of ALS (\cref{alg:basicALS}).
Notice that the cost function $f^{[d]}$ is a high degree polynomial in $\v w$ and $\m A$. 
Thus, we do not expect global convergence guarantees to the global minimum (which is generically unique at the population level due to \cref{thm:wMid}).
That said, we are still able to prove local linear convergence for finitely many samples.

\begin{theorem} \label{thm:alsconv}
    Let $\mathcal{D}$ be a conditionally-independent mixture  \eqref{eq:cond-inp-mix} with weights $\v w^* \in \Delta_{r-1}$ such that $w^*_j > 0$ for each $j$ and  Zariski-generic means $\m A^* \in \mathbb{R}^{n \times r}$.  
    Let $d_1, d_2$ be as in \cref{thm:wMid} and not both even. 
    Let $d \in \mathbb{Z}$  and $\v \gt \in \mathbb{R}_{\geq 0}^d$ satisfy $d \geq d_1, d_2$ and $\tau_{d_1}, \tau_{d_2} > 0$.  
    Let $\m V_p \in \mathbb{R}^{n \times p}$ be an i.i.d. sample of size $p$ from $\mathcal{D}$, and $\widehat{\t M}^i_p$ the corresponding $i$th sample moment \eqref{eq:sample}.
    Consider applying \cref{alg:basicALS} to minimize $f^{[d]}: \gD_{r-1}\times \R^{n\times r} \rightarrow \R_{\geq 0}$ given by
    \begin{equation} \label{eq:alsconv_cost}
        f^{[d]}(\v w, \m A; \m V_p) = \sum\nolimits_{i = 1}^d \gt_i \big{\|} \P\big{(}\sum\nolimits_{j = 1}^r w_j \v a_j^{\otimes i} - \widehat{\t M}^i_p\big{)}\big{\|}^2.
    \end{equation}
    Then the cost is non-increasing at each iteration of \cref{alg:basicALS}, and
    \begin{enumerate}  
        \item For any critical point $(\v w_c, \m A_c)$ of $f^{[d]}$ with $\v w_c$ in the interior of $\gD_{r-1}$ such that the Hessian of $f^{[d]}$ (with respect to $\v w \in \gD_{r-1}$ and $\m A$) is positive definite, \cref{alg:basicALS} converges locally linearly to $(\v w_c, \m A_c)$.
        \item For generic ground-truth weights $\v w^*$ and means $\m A^*$, there exist open neighborhoods $\mathcal{O}_i$ of $\P \t M^i$ so that if $\P \widehat{\t M}^i_p \in \mathcal{O}_i$ for all $i$, then there is a unique global minimum $(\v w_p^*, \v A_p^*)$ of $f^{[d]}$, and the Hessian is positive definite at $(\v w_p^*, \v A_p^*)$.
        \item As $p \rightarrow \infty$, $(\v w^*_p, \m A_p^*)$ converges almost surely to $(\v w^*, \m A^*)$.
    \end{enumerate}
\end{theorem}

The proof of \cref{thm:alsconv} is in \cref{apx:pf_conv}.
The ingredients are a classic result in alternating optimization \cite{bezdek2003convergence}, together with facts about polynomial \nolinebreak  mappings \nolinebreak \cite{sommese2005numerical}.

\begin{remark}
    Bounding the size of  $\cO_i$ in \cref{thm:alsconv}
    seems nontrivial at this level of generality.
    For comparison, in \cite{balakrishnan2017statistical} about a likelihood-based algorithm for mixture models, 
    the authors could control the size of the attraction basins only for specific parametric models, e.g. a mixture of two Gaussians.
\end{remark}

In the next propositions, we establish detailed rates for the modes of convergence in items 1 and 3 of \cref{thm:alsconv}.
First comes the rate of the local linear convergence.

\begin{proposition}[local linear convergence] \label{prop:linconv}
Take the setup of \cref{thm:alsconv} and assume $\widehat{\t M}_p^i \in \mathcal{O}_i$.
    Let $\m H_{f^{[d]}} \in \mathbb{R}^{(r + nr) \times (r + nr)}$ be the Hessian matrix of $f^{[d]}$ with respect to $(\v w, \m A)$ evaluated at $(\v w_p^*, \m A_p^*; \m V_p)$, with rows and columns ordered by $(\v w, \v a_1, \ldots, \v a_r)$.
    Let $\m Q \in \R^{n \times (n-1)}$ be an orthonormal basis for the space orthogonal to the all-ones vector.
    Let $\widehat{\m Q} = \operatorname{Diag}({\m Q, \m I_{nr}}) \in \mathbb{R}^{(n + nr) \times (n-1+nr)}$ and $\m H = \widehat{\m Q}^\top \m H_{f^{[d]}} \widehat{\m Q}$. 
    Run \cref{alg:basicALS} as in \cref{thm:alsconv}.  Let its iterates be $(\v w_{p}^{(t)}, \m A_{p}^{(t)})$  and define the errors by 
    $\v e_p^{(t)} := (  \v w_p^{(t)} - \v w_p^*, \m A_p^{(t)} - \m A_{p}^*) \in \mathbb{R}^{r + nr}$.  Then for $\v e_p^{(t)}$ sufficiently small,
    \begin{equation}
        \v e_p^{(t+1)} = (\m I - \m M^{-1} \m H) \v e_p^{(t)} + \cO(\|\v e_p^{(t)}\|^2),
    \end{equation}
    where $\m M$ is the block lower-triangular part of $\m H$.
    Moreover, the spectral radius obeys $\gr(\m I - \m M^{-1} \m H) < 1$.  Explicit formulas for the entries of $\m H$ are in \cref{apx:pf_conv}.
\end{proposition}
The proof is in \cref{apx:pf_conv}.  It applies \cite{bezdek2003convergence} and involves straightforward calculations.
The next result gives the convergence rate of $(\v w_p, \m A_p)$ to $(\v w^*, \m A^*)$ as $p \rightarrow \infty$.

\begin{proposition}[convergence to population] \label{prop:pconv}
    Take the setting of \cref{thm:alsconv} and assume $\widehat{\t M}_p^i \in \mathcal{O}_i$.
    Let $\m J_i \in \mathbb{R}^{{n^i} \times (r + nr)}$ 
    be the Jacobian matrix of the map $T^{i} : (\v w, \m A) \mapsto \P(\sum_{j = 1}^r w_j \v a_j^{\otimes i})$ evaluated at $(\v w^*, \m A^*)$.  
    Let
       $ \gee_p = \big{(}\sum\nolimits_{i = 1}^d \gt_i \|\widehat{\t M}^i_p - \t M^i \|^2 \big{)}^{1/2}$ 
    and define
      \begin{equation} \label{eq:sigma}
        \gs
       \, = \,
         \gs_{\min}\big{(}\sum\nolimits_{i = 1}^d \gt_i \m J_i^{\top}   \m J_i \big{)}
    \end{equation}
    where $\gs_{\min}$ is the smallest singular value.
       Then $\gs > 0$ and for $\gee_p$ sufficiently small, 
    \begin{equation} \label{eq:converge-to-population}
        \big{(}\|\v w_p^* - \v w^*\|^2 + \|\m A_p^* - \m A^*\|^2\big{)}^{1/2}
        \leq \,
        \gs^{-1/2} \gee_p + \mathcal{O}(\gee_p^2).
    \end{equation}
    Explicit formulas for the entries of $\m J_i^{\top} \m J_i$ are in \cref{apx:pf_conv}.
\end{proposition}

The proof is given in \cref{apx:pf_conv}.  It builds on the proof for Theorem~\ref{thm:alsconv} and further employs the implicit function theorem.
We do not attempt to bound the attraction basin for this polynomial optimization problem, i.e. how close is sufficiently close to guarantee these convergence properties.

    \section{Algorithm for General Means} \label{sec:gmeans}

This section is dedicated to solving for the
general means 
$\E_{X \sim \cD_j}[g(X)]$ where $g$ is a Cartesian product of functions: $g(\v x) = (g_1(x_1),\ldots,g_n(x_n))$, per our approach to nonparametric estimation \eqref{eq:fnal}.
We assume that
the means $\m A$ and weights $\v w$ have been estimated already.
Let $\v y_j = \E_{X \sim \cD_j} [g(X)]  \in \mathbb{R}^n$ and $\m Y = (\v y_1|\ldots|\v y_r) \in \R^{n\times r}$ be the unknowns we want to determine.

Based on \cref{thm:gmeans_id}, we aim to solve for $\m Y$ by solving the linear systems in \cref{prop:general-equations}.
Specifically we solve the following linear least squares for $\m Y$: 
\begin{align}
&  \min_{\m Y \in \mathbb{R}^{n \times r}} \sum\nolimits_{i = 1}^d \gt_i \big{\|} \P \big{(}\sum\nolimits_{\ell=1}^p (1/p) g(\v v_{\ell}) \otimes \v v_{\ell}^{\otimes i-1} - \sum\nolimits_{j=1}^r w_j \v y_j \otimes \v a_j^{\otimes i-1}\big{)} \big{\|}^2. \label{eq:general-linear} 
\end{align}
It turns out that we can solve this linear system efficiently (in a tensor-free way) using methods developed in \cref{sec:mstep}.
Notice that the linear system \eqref{eq:general-linear} is block-diagonal in the rows of $\m Y$.
Without loss of generality, suppose we are solving for the first row $\v y^1$. 
To take advantage of the learned weights and means, consider applying $g_1$ to the first row of the data $\m V$ to obtain $\v V'$.
Obviously, the rows of mean matrix for $\v V'$ are now $(\v y^1,\v a^2,\ldots,\v a^n)$.
Thus, we compute $\v y^1$ by running the row update on the mean matrix of the data set $\v V'$ with rows $2$ through $n$ in its mean matrix fixed to $(\v a^2 | \ldots | \v a^n)$.
This gives the following result for the normal equations \nolinebreak of  \nolinebreak \eqref{eq:general-linear}:

\begin{proposition} \label{prop:solve_gmeans}
    Define $\m G^{A^{(k)}\!, A^{(k)}}_s$ and $\m G^{A^{(k)},  V^{(k)}}_s$ as in \cref{prop:row_update}.
    Let 
    \begin{equation*}
        \hat{\m L}, \hat{\v r} \leftarrow \small\prepnormeqn(\{\m G^{A^{(k)},  A^{(k)}}_s\}_{s = 1}^d,~\{\m G^{ A^{(k)},  V^{(k)}}_s\}_{s = 1}^d,~\frac{g_k(\v v^k)}{p},~d-1,~\!(2\gt_2,\ldots,d\gt_d)),
    \end{equation*}
    and $\m L = \hat{\m L} + \gt_1$, $\v r = \hat{\v r} + \gt_1  \mu_k$, where $\v \mu$ is the mean vector of $g(\m V)$ ($g$ is applied to each column of $\m V$).
    Then the normal equations for $\v \gb^k = \v w * \v y^k \in \mathbb{R}^r$ in \eqref{eq:general-linear} are given by $\m L \v \gb^k = \v r$.
\end{proposition}

This is a natural extension of \cref{prop:row_update} on updating the rows of the mean matrix.  Indeed, \cref{prop:row_update} is the special case of $g(\v x) = \v x$. 
The proof of \cref{prop:solve_gmeans} 
is identical to that of \cref{prop:row_update}, and we do not reiterate it.
We summarize the procedure of solving for the general means in \cref{alg:gmom}.

\begin{algorithm}
    \begin{algorithmic}[1]
        \caption{Solve for $\v y_j = \E_{X \sim\cD_j} [g(X)]$ given means and weights} \label{alg:gmom} 
        \Function{solveGeneralMean}{function $g$, $\{\m G^{A, A}_s\}_{s = 1}^d$, $\{\m G^{A, V}_s\}_{s = 1}^d$, data $\m V$, weights $\v w$, means $\m A$, order $d$, hyperparameters $\v \gt$}
        
            \State Compute $\m V_g \gets g(\m V)$ and column mean $\v \mu$ of $\m V_g$
            \For{$k = 1, \ldots, n$} \Comment\textbf{solve $\v y^k$}
                \For{$s = 1, ..., d$}
                    \Comment \textbf{rank-1 deflation}
                    \State 
                    $\m G^{A^{(k)}, A^{(k)}}_s \gets \m G^{A, A}_s - (\v a^k)^{*s} \left((\v a^k)^{*s}\right)^\top$
                    \State 
                    $\m G^{A^{(k)}, V^{(k)}}_s \gets \m G^{A, V}_s - (\v a^k)^{*s} \left((\v v^k)^{*s}\right)^\top$
                \EndFor
                \State $\v \pi \gets \m V_g(k, :)/p$
                \State $\m Y(k, :) \gets \solverow(
                    k,
                    \{\m G^{A^{(k)}, A^{(k)}}_s\}_{s = 1}^d, 
                    \{\m G^{A^{(k)}, V^{(k)}}_s\}_{s = 1}^d,
                    \v w, \v \pi, \v \mu, d, \v \gt
                )$
            \EndFor
        \EndFunction
        \Return $\m Y$
    \end{algorithmic}
\end{algorithm}

\begin{lemma} \label{prop:gmom_cost}
    Assume $d = \cO(1)$ and $g(\m V)$ is computed in $\cO(np)$ time. 
    Then \cref{alg:gmom} computes the general means $\m Y$  in $\cO(npr + nr^3)$ flops and $\cO(r(n + p))$ storage.
    Moreover, \cref{alg:gmom} can compute $I$ different general means $\m y^{(i)}_j = \E_{X \sim \cD_j}[g^{(i)}(X)]$ for $i = 1,\ldots,I$ in $\cO(Inpr + nr^3)$ time. 
\end{lemma}

\begin{proof} 
    The first part ($I = 1$) is clear.  For a single function $g$, \cref{alg:gmom} has the same order of flop count as the mean update \cref{alg:Mstep}. Hence, the first assertion follows.
    For the second part, the complexity appears to be $\cO(I(npr + nr^3))$. 
    However we can reduce the time from $\cO(Inr^3)$ down to $\cO(nr^3)$, because the $\m L$ matrix in the normal equation is identical for all $I$ general mean solves (it only depends on $\m A$). Thus, we only need $\cO(nr^3)$ time to solve $\m L \v \gb^k = \v r$ for $I$ different $\v r$ vectors.
\end{proof}

Parallel to \cref{prop:invertL_w} and \ref{prop:invertL_a}, we prove that the linear system \eqref{eq:general-linear} is generically full rank and thus has a unique solution (independent of the function $g$).

\begin{proposition} \label{prop:full-rank-general}
Assume $\gt_d > 0$ and $\m A^{(k)}$ is Zariski-generic, where $\m A^{(k)}$ denotes $\m A$ with the $k$th row removed.  Then the matrix $\m L \in \mathbb{R}^{r \times r}$ for the normal equations in \cref{prop:solve_gmeans} is invertible, provided $r \leq \binom{n-1}{d-1}$.
\end{proposition}

In particular,
taking $g(\v x) = \ind{\v x \leq \v t}$ lets us solve for cumulative distribution function of the components $\mathcal{D}_j$ in the mixture.
The proof is identical to that of \cref{prop:invertL_a} so we do not repeat it.

\begin{remark}
We can easily impose convex constraints on the each row of the general means matrix $\m Y$. 
For example, for estimation of components' cumulative density function we want $\v y^k \in [0, 1]^r$, whereas for second moments, we want $\v y^k \geq (\v a^k)^{*2}$.
Including such constraints, each solve of $\v y^k$ 
becomes a convex quadratic program. 
They are handled similarly to how we dealt with the simplex constraint on $\v w$ in \cref{sec:wstep}.
\end{remark}

    \section{Improving the Baseline ALS Algorithm}\label{sec:improve}
We present practical recommendations and enhancements to the baseline ALS \cref{alg:basicALS} for the means and weights.
Since we are optimizing a nonconvex cost function in ALS, the main goals are to accelerate the convergence as well as avoid bad numerical behavior and bad local minima.
We list the recommendations and brief ideas here.
More details can be found in \cref{apx:practice}.  

\begin{itemize}
    \item \textbf{Anderson Acceleration (AA).} This is a standard acceleration routine for optimization algorithms to accelerate terminal convergence. 
    See \cite{walker2011anderson} for details.
    We use a short descent history, instead of just the current state, to predict a better descent.
    Gradient information is necessary, and we show the gradients can be evaluated implicitly with the same order of cost (see \cref{apx:anderson}).

    \item \textbf{A Drop-One Procedure.} To avoid bad local minima, we alter the cost function at each iteration by changing the $\v \gt$ coefficients.  This changes the critical points.
    However identifiability theory (\cref{subsec:identify}) guarantees that the only common global minimum for (almost) all choices of $\v \gt$ is our means and weights.
    To optimize the efficiency, we set some $\gt_i$ to zero in each iteration, where $i$ is chosen according to the gradient information (see \cref{apx:dropone}).

    \item \textbf{Blocked ALS Sweep.} Instead of solving one row of the mean matrix $\m A$ at a time, we can solve for $m > 1$ rows at once when the other $n-m$ rows fixed.
    In many practical problems $r \ll n^{\floor{d/2}}$, so we still have an identifiability guarantee while holding only $n-m$ rows fixed.
    This accelerates the ALS row sweep, and by randomly choosing the $m$ rows to update, it also introduces randomness to avoid local minima in the descent (see \cref{apx:blockALS}).

    \item \textbf{Restricting the Parameters.} At the beginning of the descent, we regularize the weights by enforcing $w_j \geq c > 0$ for some choice of $c$ and also the means by restricting them to the circumscribing cube of the data cloud.
    This can avoid pathological iterates with bad initializations (see \cref{apx:restr_param}).
\end{itemize}

We apply drop-one, blocked ALS sweep, and parameter restriction at early stages of the descent (the warm-up stage), and AA in the main stage.
Putting them together we obtain the improved version of ALS called \nolinebreak \ALS\!\!\! summarized in \cref{alg:finalALS}.  We implemented \cref{alg:finalALS} for the large-scale tests in the next section.

\begin{algorithm}
    \begin{algorithmic}[1]
        \caption{Full \ALS algorithm for solving means and weights}
            \label{alg:finalALS}
            \Function{ImprovedSolveMeanAndWeight}{data $\m V$, initialization $(\v w, \m A)$, order $d$, hyperparameter $\v \gt$, number of warm-up steps  $n_1$}
            \State $n_{\text{iter}} \gets 0$
            \State Compute $\{\m G^{A, A}_s\}_{s = 1}^d$, $\{\m G^{A, V}_s\}_{s = 1}^d$
            \While{not converged}
                \If{$n_{\text{iter}} < n_1$} 
                \Comment \textbf{warm-up stage}
                    \State Run Drop-One. Select $i$ and set $\gt_{i} \gets 0$ 
                    \State Do \parbox[t]{0.8\linewidth}{blocked ALS sweep to update $\m A$ by randomly grouping the rows. \\Restrict each update to the circumscribing cube of data.}
                    \State Reset $\gt_i$ to its original value.
                    \State Update weights by $\wstep$ enforcing   lower bound $w_j \geq c$.
                \Else
                \Comment \textbf{main stage}
                    \State Update means by $\mstep$
                    \State Update weights by $\wstep$
                    \State Run AA to further update the weights and means
                \EndIf
                \State $n_{\text{iter}} \gets n_{\text{iter}} + 1$
                \vspace{1ex}
            \EndWhile
            \EndFunction
        \Return $\v w$, $\m A$
    \end{algorithmic}
\end{algorithm}

\begin{remark}
    As mentioned above, we show in \cref{apx:deriv} that the gradients of $f^{[d]}$ can be evaluated efficiently in a tensor-free way, with the same order of computation and storage costs as in ALS.
    Thus, an alternative optimization approach is to use gradient descent based algorithms on $f^{[d]}$.
    We compared ALS with the BFGS algorithm, and ALS was faster in our tests.
\end{remark}

    \section{Numerical Experiments} \label{sec:exp}

This section demonstratew the performance of \ALS (\cref{alg:finalALS}) and the general means solver (\cref{alg:gmom}) on various numerical tests.\footnote{Implementations of our algorithms in Python are available at \url{https://github.com/yifan8/moment_estimation_incomplete_tensor_decomposition}} 
In \cref{sec:density}, we apply the pipeline outlined in \cref{sec:contrib} to estimate the component distributions using our algorithm.
In \cref{sec:mix_param_test}, we test the algorithm on mixtures of gamma and Bernoulli distributions.
We also test a ``heterogeneous" mixture model, where features in the data come from different distributions.
We generate data from these distributions  but do not use knowledge of the distribution type.  Comparisons are to the EM algorithm which knows these distribution types.
In \cref{sec:xfel}, \ALS is tested on a larger-scale synthetic dataset modeling a biomolecular imaging application, X-Ray Free Electron Lasers (XFEL) \cite{huang2007review}.  
The task is computing denoised class averages from noisy XFEL images, and the simulated noise is from a mixture of Poissons which is standard in XFEL literature \cite{ayyer2016dragonfly}.   
We compare \ALS, running nonparametrically, to the EM algorithm knowing the distribution information.
In \cref{sec:usps}, we demonstrate the robustness of our approach
on a real dataset \cite{usps}.  
The dataset concerns handwritten digit recognition, where conditional independence in the mixture components definitely fails.  
Nonetheless \ALS does a reasonable job, with results  similar to those obtained by $\textup{k-means}^{++}$.

For the tests on parametric mixtures  
we use error metrics defined in terms of the weights, means and higher moments computed from the sample.  
For means, this is
\begin{equation} \label{eq:error-metric}
    \texttt{error\_means} \,\, = \,\, \min_{\Pi } \, \| \m{{A}} \, \Pi - \m{A}_{\textup{sample}}  \|_F^2  \big{/}  \| \m{A}_{\textup{sample}} \|_F^2,
\end{equation} 
where $\m A$ is the computed mean matrix, $\m A_{\text{sample}}$ is the true mean matrix of the sample, and $\gP$ ranges over all column permutations.
The metrics used for the weights and higher moments are analogous (except that the permutation matrix $\Pi$ is fixed as the minimizer for the means in \eqref{eq:error-metric}).  
These error metrics measure how well algorithms perform conditional on the given data, excluding the influence of the random sampling which no algorithm can control anyways.

In the synthetic experiments, we assume the ground truth $r$ is known. 
We always center the data to mean 0 and scale it so that $\var(\v v^k) = 1$ for all $k \in [n]$ to improve numerical stability.
The hyperparameters $\gt_i$ in $f^{[d]}$ are chosen to be inversely proportional to the number of entries in the off-diagonal of $i$th residual tensor
\begin{equation*}
    \gt_i = (n-i)!/n!.
\end{equation*}
As a stopping criterion,
we require the relative $\ell_2$-changes in $\v w$ and $\m A$ to both drop below $xtol = 10^{-4}$ as the stopping criterion.  
Other hyperparameters in our algorithms are held constant, and not tuned for performance; their values are tabulated in \cref{apx:hyper}.
The EM algorithm for non-Gaussian mixtures we compare to is from Python package \textit{pomegranate}.\footnote{\url{https://pomegranate.readthedocs.io}, version 0.14.6, see \cite{schreiber2017pomegranate}}
All experiments are performed on a personal computer with an Intel Core i7-10700K CPU and enough memory.
Our algorithms are implemented in Python 3.8 with \textit{scipy}\footnote{\url{https://scipy.org}, version 1.8.0} and \textit{numpy}.\footnote{\url{https://numpy.org}, version 1.20.3}
All quadratic programs 
are solved by the Python package \textit{qpsolvers}.\footnote{\url{https://pypi.org/project/qpsolvers}, version 1.6.1}

From the experimental results, we see a clear takeaway:  
\textit{the ALS framework is at lower risk of converging to a bad local minimum than EM}.
Although \ALS tends to be slower, if we run EM with multiple initializations to have a total runtime on par with \ALS \! then our algorithm still exhibits better errors for bigger problems.

\subsection{Estimation of Component Distributions}\label{sec:density}

We take $n = 10$ and $r = 3$.  Each component follows a diagonal Gaussian distribution.
The weights $\v w$ are sampled uniformly on $[1, 5]^r$ and then normalized.
Mean vectors $\v a_j$ are sampled from the standard normal density in $\R^n$.
The vector of standard deviations in each component $\v \gs_j$ is sampled from the standard normal distribution in $\R^n$ and then absolute values are taken.
We cap the standard deviation from below at 0.1 to avoid pathological behavior.
Forty thousand samples are generated.
We use first $d = 4$ moments in the \ALS algorithm.  The goal is to then use \cref{alg:gmom} to estimate cumulative density functions (cdf) of the component distributions.

First we solve for the means and weights by \ALS.  Then we compute the standard deviations by solving the second moment $\v m_j^2$ from general means $\E_{X\sim\cD_j}[X^{*2}]$.
We impose the constraint $\v m_j^2 \geq (\v a_j)^{*2}$.
Last we estimate the cumulative distribution functions of each feature in each component 
by evaluating the general means 
\begin{gather*}
    \Big(\pr_{X\sim\cD_j}(X_1 \leq t_1),\ldots,\pr_{X\sim\cD_j}(X_n \leq t_n)\Big) = \E_{X\sim\cD_j}[g(X)],\\
    \text{where} \ \ 
    g(\v x) = (\ind{x_1\leq t_1},\ldots,\ind{x_n\leq t_n}).
\end{gather*}
Denote the computed means and standard deviations from \cref{alg:finalALS} and the general means solver by $\hat{\v a}_j$ and $\hat{\v \gs}_j$.
We evaluate the cdf of the $j$th component at $\v t = \hat{\v a}_j + k \hat{\v\gs}_j$ for $k = -2, -1.5,\ldots, 1.5, 2$.

Figure \ref{fig:density} shows the results including both the accuracy of estimating the means and standard deviations and the accuracy of the distribution estimation.
We see that $\v \a_j$, $\v \gs_j$, and the distributions are accurately estimated.  Notably also, the grid on which distribution estimations performed is adapted to the shape of the cdf. 

\begin{figure}[!htbp]
    \centering
    \includegraphics[width = \linewidth]{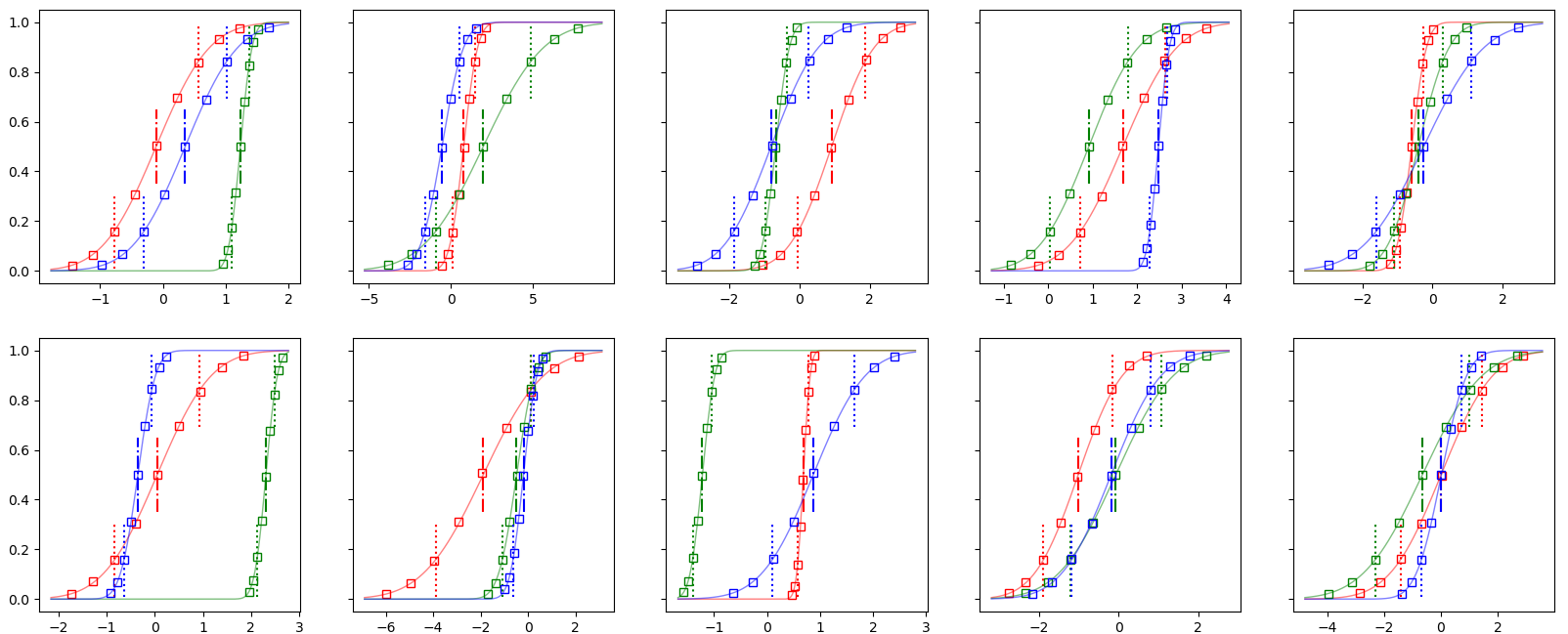}
    \caption{Results on 10 coordinates of the 3 Gaussian mixture components. Each subplot corresponds to a single coordinate (feature). Colors correspond to different components.
    The curves indicate the true cdf. 
    In the $i$th subplot,
    the dashed and dotted vertical lines indicate the location of the true means $(\v a_j)_i$ and $(\v a_j)_i \pm (\v \gs_j)_i$. 
    This defines the grid on which localized density estimation is performed.
    The squares are placed at $(\hat{\v a}_j)_i + k (\hat{\v \gs}_j)_i$, with $y$-value being the computed cdf values.}  
    \label{fig:density}
\end{figure}

\subsection{Performance on Parametric Mixture Models} \label{sec:mix_param_test}

For each parametric model and $(n,r)$ pair below, we run $20$ independent simulations. 
Ground-truth parameters are randomly generated per simulation. 
The weights $\v w$ are sampled uniformly on $[1, 5]^r$ and then normalized, and the other parameters are sampled by a procedure described below for each model.
Then $p=20000$ data points are drawn from the resulting mixture.  
First we apply \ALS using the first $d = 4$ moments of the data to  compute the means and weights.  
Then we use \cref{alg:gmom} to compute the second moments of the mixture components (except in \cref{sec:cubemix} where such information is redundant).
In the second moment solve, we impose the constraint $\v m_j^2 \geq \v a_j^{*2} + 10^{-4}$ for all $j \in [r]$, which makes the linear solve $\m L^{-1} \v r$ in \cref{alg:gmom} into a quadratic program.

We compare \ALS to the EM algorithm.
For each EM run, we initialize with 10 $\textup{k-means}^{++}$ runs, and use the best result as the starting point for EM.
We call the resulting procedure the ``EM1'' algorithm.
Since EM is known to converge fast, but is prone to local minima, when EM1 runs much faster than our algorithm, 
we repeat EM1 ten times and keep the best result.
This is called the ``EM10'' algorithm.
The tolerance level of the EM algorithm is left at the default value $10^{-1}$ in \textit{pomegranate}.
The accuracy of the results does not appear to be sensitive to the tolerance.

\subsubsection{Mixture of Bernoullis} \label{sec:cubemix}

The Bernoulli mixture model is a natural way of encoding true or false responses and other binary quantities. 
It is studied in theoretical computer science under the name mixture of subcubes \cite{chen2019beyond}.
A Bernoulli in $\mathbb{R}^n$ with independent entries has probability mass function:
\begin{equation} \label{eq:bernoulli}
    p(X = \v x | \v \mu) \, = \, \prod\nolimits_{i = 1}^n \left( \ind{x_i = 1} \mu_i + \ind{x_i = 0} (1 - \mu_i) \right)
\end{equation}
where $\v \mu \in [0,1]^n$. 
The true mean vectors $\v \mu$ are sampled i.i.d. uniformly from the cube $[0, 1]^n$.
EM10 with the \pom implementation is used for comparison.

The results are shown in \cref{tab:subcubeALS}.  Here the timing between \ALS and EM10 is comparable.
We see that the \ALS algorithm shows consistently good accuracy om high dimensions, whereas EM10 is less likely to converge to the true solutions.

\begin{table}[htbp] 
    \setlength{\tabcolsep}{4.5pt}
    \ssmall
    \caption{Results for the Bernoulli mixture tests, comparing  \cref{alg:finalALS} ($p=20000, d=4$)  to \pom EM. 
    }
    \label{tab:subcubeALS}
    \begin{center}
        \begin{tabular}{c|c|cccc|cccc|cccc}
            \hline

            \multirow{3}{*}{n}  & \multirow{3}{*}{r} & \multicolumn{4}{c|}{Weight (\%)}                                                                                              & \multicolumn{4}{c|}{Mean (\%)}                                                                                                & \multicolumn{4}{c}{Time (s)}                                                             \\ \cline{3-14} 
                                &                    & \multicolumn{2}{c|}{Ours}                                   & \multicolumn{2}{c|}{EM10}                                       & \multicolumn{2}{c|}{Ours}                                   & \multicolumn{2}{c|}{EM10}                                       & \multicolumn{2}{c|}{Ours}                            & \multicolumn{2}{c}{EM10}          \\ \cline{3-14} 
                                &                    & Ave                            & \multicolumn{1}{c|}{Worst} & Ave                            & Worst                          & Ave                            & \multicolumn{1}{c|}{Worst} & Ave                            & Worst                          & Ave            & \multicolumn{1}{c|}{Worst}          & Ave             & Worst           \\ \hline
            \multirow{3}{*}{15} & 3                  & 0.51                           & \multicolumn{1}{c|}{1.88}  & \textbf{0.41} & \textbf{1.03} & 0.57                           & \multicolumn{1}{c|}{1.11}  & \textbf{0.37} & \textbf{0.77} & \textbf{2.58}  & \multicolumn{1}{c|}{\textbf{3.65}}  & 3.66            & 5.69            \\ \cline{2-14} 
                                & 6                  & 2.20                           & \multicolumn{1}{c|}{3.85}  & \textbf{1.32} & \textbf{2.27} & 1.47                           & \multicolumn{1}{c|}{2.86}  & \textbf{0.98} & \textbf{1.48} & \textbf{7.29}  & \multicolumn{1}{c|}{15.34}          & 8.87            & \textbf{13.51}  \\ \cline{2-14} 
                                & 9                  & 3.27                           & \multicolumn{1}{c|}{6.69}  & \textbf{2.27} & \textbf{3.31} & 2.08                           & \multicolumn{1}{c|}{3.23}  & \textbf{1.45} & \textbf{2.35} & 16.38          & \multicolumn{1}{c|}{36.36}          & \textbf{16.32}  & \textbf{25.65}  \\ \hline
            \multirow{3}{*}{30} & 6                  & 0.63                           & \multicolumn{1}{c|}{1.45}  & \textbf{0.27} & \textbf{0.63} & 0.72                           & \multicolumn{1}{c|}{0.96}  & \textbf{0.30} & \textbf{0.51} & \textbf{7.90}  & \multicolumn{1}{c|}{11.76}          & 8.93            & \textbf{11.54}  \\ \cline{2-14} 
                                & 12                 & 1.74                           & \multicolumn{1}{c|}{3.41}  & \textbf{0.58} & \textbf{0.94} & 1.45                           & \multicolumn{1}{c|}{1.78}  & \textbf{0.57} & \textbf{0.72} & 26.46          & \multicolumn{1}{c|}{38.28}          & \textbf{25.59}  & \textbf{34.10}  \\ \cline{2-14} 
                                & 18                 & 3.02                           & \multicolumn{1}{c|}{\textbf{5.58}}  & \textbf{2.22} & 8.55                           & \textbf{2.12} & \multicolumn{1}{c|}{\textbf{2.48}}  & 4.43                           & 19.59                          & \textbf{46.51} & \multicolumn{1}{c|}{65.30}          & 46.61           & \textbf{60.08}  \\ \hline
            \multirow{3}{*}{50} & 10                 & 0.65                           & \multicolumn{1}{c|}{0.99}  & \textbf{0.13} & \textbf{0.27} & 0.85                           & \multicolumn{1}{c|}{1.00}  & \textbf{0.13} & \textbf{0.18} & 29.15          & \multicolumn{1}{c|}{39.16}          & \textbf{25.36}  & \textbf{37.08}  \\ \cline{2-14} 
                                & 20    & 1.89  & \multicolumn{1}{c|}{\textbf{2.81}}  & \textbf{1.19} & 7.85 & \textbf{1.65} & \multicolumn{1}{c|}{\textbf{1.87}} & 2.65 & 16.73 & \textbf{68.39} & \multicolumn{1}{c|}{\textbf{95.60}} & 78.01  & 101.60          \\ \cline{2-14}
                                & 30    & \textbf{3.64} & \multicolumn{1}{c|} {\textbf{5.67}}  & 3.72  & 9.20  & \textbf{2.42} & \multicolumn{1}{c|}{\textbf{3.14}}  & 9.28  & 15.69 & 149.3& \multicolumn{1}{c|}{265.41}& \textbf{144.47} & \textbf{179.03} \\ \hline
        \end{tabular}
    \end{center}
\end{table}

\subsubsection{Mixture of Gammas} \label{sec:gammamix}

The gamma mixture model, including mixture of exponentials, is a popular choice for modeling positive continuous signals \cite{wiper2001mixtures}.
A gamma-distributed random variable in $\mathbb{R}^n$ with independent entries has density: 
\begin{equation}
    p(X = \v x | \v k, \v \theta) \propto \prod\nolimits_{i = 1}^n x_i^{k_i-1} e^{-x_i/\theta_i}
\end{equation}
where $\v k \in (\mathbb{R}_{>0})^n$ and $\v \theta \in (\mathbb{R}_{>0})^n$ (using the shape-scale parametrization convention).
The scale parameters $\v \theta$ are sampled uniformly from $[0.1, 5]^n$, and the shape parameters $\v k$ are sampled  uniformly from $[1, 5]^n$.
As opposed to the previous tests, the EM algorithm takes a longer time to run compared to the \ALS algorithm for larger problems.
In light of this, we only compare to EM1.

The results are listed in \cref{tab:gamma}. 
The EM algorithm is struggling in this case even in low-dimensional problems.
In higher dimensions, the \ALS algorithm offers consistently good accuracy as well as better running times.

\begin{table}[htbp]
    \footnotesize
    \caption{Results for the Gamma mixture tests, comparing \cref{alg:finalALS} and \cref{alg:gmom} ($p=20000, d=4$) to \pom EM.
    }
    \label{tab:gamma}
    \begin{center}

    \begin{tabular}{c|c|cccc||cccc}
        \hline
        \multirow{3}{*}{n}  & \multirow{3}{*}{r} & \multicolumn{4}{c||}{Weight (\%)}                                                   & \multicolumn{4}{c}{Mean (\%)}                                                           \\ \cline{3-10} 
                            &                    & \multicolumn{2}{c|}{Ours}                           & \multicolumn{2}{c||}{EM1}      & \multicolumn{2}{c|}{Ours}                              & \multicolumn{2}{c}{EM1}         \\ \cline{3-10} 
                            &                    & Ave           & \multicolumn{1}{c|}{Worst}         & Ave           & Worst         & Ave            & \multicolumn{1}{c|}{Worst}           & Ave            & Worst          \\ \hline
        \multirow{3}{*}{15} & 3                  & \textbf{0.36} & \multicolumn{1}{c|}{\textbf{1.13}} & 4.00          & 41.02         & \textbf{0.41}  & \multicolumn{1}{c|}{\textbf{0.60}}   & 4.74           & 51.96          \\ \cline{2-10} 
                            & 6                  & \textbf{0.85} & \multicolumn{1}{c|}{\textbf{1.60}} & 5.57          & 26.46         & \textbf{0.77}  & \multicolumn{1}{c|}{\textbf{1.15}}   & 10.17          & 39.56          \\ \cline{2-10} 
                            & 9                  & \textbf{1.44} & \multicolumn{1}{c|}{\textbf{2.95}} & 8.21          & 23.85         & \textbf{1.13}  & \multicolumn{1}{c|}{\textbf{1.44}}   & 15.85          & 42.95          \\ \hline
        \multirow{3}{*}{30} & 6                  & 0.35          & \multicolumn{1}{c|}{0.56}          & \textbf{0.00} & \textbf{0.02} & 0.47           & \multicolumn{1}{c|}{0.61}            & \textbf{0.00}  & \textbf{0.02}  \\ \cline{2-10} 
                            & 12                 & \textbf{0.83} & \multicolumn{1}{c|}{\textbf{1.52}} & 8.87          & 23.04         & \textbf{0.80}  & \multicolumn{1}{c|}{\textbf{0.97}}   & 20.48          & 33.51          \\ \cline{2-10} 
                            & 18                 & \textbf{1.26} & \multicolumn{1}{c|}{\textbf{1.70}} & 15.65         & 34.64         & \textbf{1.11}  & \multicolumn{1}{c|}{\textbf{1.34}}   & 26.70          & 32.46          \\ \hline
        \multirow{3}{*}{50} & 10                 & \textbf{0.42} & \multicolumn{1}{c|}{\textbf{0.63}} & 7.68          & 24.86         & \textbf{0.52}  & \multicolumn{1}{c|}{\textbf{0.61}}   & 15.34          & 29.75          \\ \cline{2-10} 
                            & 20                 & \textbf{0.77} & \multicolumn{1}{c|}{\textbf{1.00}} & 13.03         & 27.54         & \textbf{0.89}  & \multicolumn{1}{c|}{\textbf{1.04}}   & 26.80          & 35.23          \\ \cline{2-10} 
                            & 30                 & \textbf{1.13} & \multicolumn{1}{c|}{\textbf{1.44}} & 13.46         & 29.21         & \textbf{1.19}  & \multicolumn{1}{c|}{\textbf{1.34}}   & 26.18          & 34.94          \\ \hline\hline
        \multirow{3}{*}{n}  & \multirow{3}{*}{r} & \multicolumn{4}{c||}{2nd Moment (\%)}                                               & \multicolumn{4}{c}{Time (s)}                                                            \\ \cline{3-10} 
                            &                    & \multicolumn{2}{c|}{Ours}                           & \multicolumn{2}{c||}{EM1}      & \multicolumn{2}{c|}{Ours}                              & \multicolumn{2}{c}{EM1}         \\ \cline{3-10} 
                            &                    & Ave           & \multicolumn{1}{c|}{Worst}         & Ave           & Worst         & Ave            & \multicolumn{1}{c|}{Worst}           & Ave            & Worst          \\ \hline
        \multirow{3}{*}{15} & 3                  & \textbf{0.80} & \multicolumn{1}{c|}{\textbf{1.34}} & 6.77          & 73.64         & 2.23           & \multicolumn{1}{c|}{2.98}            & \textbf{0.80}  & \textbf{2.47}  \\ \cline{2-10} 
                            & 6                  & \textbf{1.36} & \multicolumn{1}{c|}{\textbf{2.06}} & 13.61         & 57.15         & 5.87           & \multicolumn{1}{c|}{13.79}           & \textbf{2.42}  & \textbf{11.58} \\ \cline{2-10} 
                            & 9                  & \textbf{2.00} & \multicolumn{1}{c|}{\textbf{2.98}} & 22.05         & 68.94         & 14.11          & \multicolumn{1}{c|}{45.72}           & \textbf{5.15}  & \textbf{15.87} \\ \hline
        \multirow{3}{*}{30} & 6                  & 0.92          & \multicolumn{1}{c|}{1.26}          & \textbf{0.33} & \textbf{0.39} & 7.57           & \multicolumn{1}{c|}{19.49}           & \textbf{1.98}  & \textbf{2.51}  \\ \cline{2-10} 
                            & 12                 & \textbf{1.49} & \multicolumn{1}{c|}{\textbf{1.82}} & 27.32         & 43.93         & 21.52          & \multicolumn{1}{c|}{\textbf{28.69}}  & \textbf{21.22} & 51.54          \\ \cline{2-10} 
                            & 18                 & \textbf{2.06} & \multicolumn{1}{c|}{\textbf{2.38}} & 36.52         & 50.61         & \textbf{39.93} & \multicolumn{1}{c|}{\textbf{59.87}}  & 43.19          & 66.99          \\ \hline
        \multirow{3}{*}{50} & 10                 & \textbf{1.02} & \multicolumn{1}{c|}{\textbf{1.43}} & 20.66         & 41.00         & \textbf{23.07} & \multicolumn{1}{c|}{\textbf{29.82}}  & 32.11          & 115.16         \\ \cline{2-10} 
                            & 20                 & \textbf{1.71} & \multicolumn{1}{c|}{\textbf{2.05}} & 35.40         & 47.53         & \textbf{51.91} & \multicolumn{1}{c|}{\textbf{72.44}}  & 123.54         & 245.30         \\ \cline{2-10} 
                            & 30                 & \textbf{2.24} & \multicolumn{1}{c|}{\textbf{2.67}} & 35.68         & 48.43         & \textbf{93.29} & \multicolumn{1}{c|}{\textbf{162.87}} & 167.09         & 327.55         \\ \hline
    \end{tabular}
    \end{center}
\end{table}

\subsubsection{A Heterogeneous Mixture Model} \label{sec:generalmix}

An important application is to estimate the weights, means, and other quantities of interest from a dataset composed of different data types. 
For example, a dataset can have some features encoded by booleans (e.g. responses to True/False questions), some by elements of a finite set (e.g. ranking or scoring), others that are continuous (e.g. Gaussian), etc.

We illustrate this by
setting $n = 40, r = 20$ and using four mixture types: Bernoulli; discrete distribution (supported on $\{1, 2, 3, 4, 5\}$); Gaussian; and Poisson.  We allocate ten features to each data type. 
The parameters are generated as follows:  
\begin{itemize}\setlength\itemsep{0.2em}
    \item For the Bernoulli mixture \eqref{eq:bernoulli}, the means $\v \mu$  are uniformly from $[0, 1]^{10}$.
    \item For the discrete mixture,  one discrete distribution has mass function:
    \begin{equation}
        p(X = \v x | \m \alpha) = \prod\nolimits_{i = 1}^{10}
        \big(\sum\nolimits_{k = 1}^5\ind{x_i = k}\alpha_{ik}\big).
    \end{equation}
    The different matrices $\m \alpha$ are drawn uniformly from $[0,1]^{10 \times 5}$ and then normalized to be row-stochastic.
    \item For the Gaussian mixture, the means $\v \mu$ are sampled from a standard Gaussian, and the standard deviations $\v \sigma$ are uniformly from  $[0, \sqrt{10}]^{10}$;
    \item For the Poisson mixture,  one Poisson distribution has density:
    \begin{equation} \label{eq:poisson}
p(X = \v x | \v \lambda ) \, = \, \prod\nolimits_{i = 1}^{10} \frac{\lambda_i^{x_i} e^{-\lambda_i}}{x_i!}.    \end{equation}
    The mean vectors $\v \lambda$
   are sampled uniformly from  $[0, 5]^{10}$.
\end{itemize}

We benchmark against \pom EM, which is given each feature's data type. 
The results are in \cref{tab:genALS}. 
With 10 descents,  EM10 takes a longer time than \ALS, and there are five out of the twenty simulations where EM10 gives a poor answer (the relative error in the means or the weights exceeds 10\%).
By contrast, the \ALS algorithm is more reliable, giving a much better worst-case performance.

\setlength{\tabcolsep}{2.8pt}
\begin{table}[htbp]
    \begin{center}
    \footnotesize
    \caption{Results for the heterogeneous mixture tests, comparing  \cref{alg:finalALS} and \cref{alg:gmom} ($p= 20000, n=40, r =20, d=4$)  to \pom EM.
    }
    \label{tab:genALS}
    \begin{tabular}{cccc||cccc}
        \hline
        \multicolumn{4}{c||}{Weight (\%)}                                               & \multicolumn{4}{c}{Mean (\%)}                                                    \\ \hline
        \multicolumn{2}{c|}{Ours}                           & \multicolumn{2}{c||}{EM10} & \multicolumn{2}{c|}{Ours}                              & \multicolumn{2}{c}{EM10} \\ \hline
        Ave           & \multicolumn{1}{c|}{Worst}         & Ave             & Worst   & Ave            & \multicolumn{1}{c|}{Worst}           & Ave         & Worst      \\ \hline
        3.41          & \multicolumn{1}{c|}{\textbf{5.15}} & \textbf{2.13}   & 10.44   & \textbf{2.54}  & \multicolumn{1}{c|}{\textbf{3.10}}   & 4.49        & 19.93      \\ \hline\hline
        \multicolumn{4}{c||}{2nd Moment (\%)}                                           & \multicolumn{4}{c}{Time (s)}                                                     \\ \hline
        \multicolumn{2}{c|}{Ours}                           & \multicolumn{2}{c||}{EM10} & \multicolumn{2}{c|}{Ours}                              & \multicolumn{2}{c}{EM10} \\ \hline
        Ave           & \multicolumn{1}{c|}{Worst}         & Ave             & Worst   & Ave            & \multicolumn{1}{c|}{Worst}           & Ave         & Worst      \\ \hline
        \textbf{3.27} & \multicolumn{1}{c|}{\textbf{4.06}} & 5.76            & 23.05   & \textbf{92.58} & \multicolumn{1}{c|}{\textbf{173.53}} & 229.61      & 353.14     \\ \hline
    \end{tabular}
    \end{center}
\end{table}
\setlength{\tabcolsep}{6pt}

We also tried to benchmark with the semiparametric EM algorithm proposed by Chauveau \cite{chauveau2016nonparametric}, using its implementation in the R package \textit{mixtools}.\footnote{\url{https://cran.r-project.org/web/packages/mixtools/index.html}, see \cite{mixtools}}
However, the R function was not able to finish the solve in a comparable amount of time on this setup.
On the smaller dimensions  $n = 8$, $r = 4$, $p = 5000$, it  took $\sim20$ seconds per  iteration.

\subsection{Application: XFEL} \label{sec:xfel}
We test the ALS framework on a simulation of a scientific application.  
In microscopy, one wants to compute the 3D map of a molecule from several noisy 2D images.  
Typically this starts with a step of 2D image clustering and denoising, known as class averaging.  
This is well-studied for cryo-electron microscopy \cite{scheres2015semi}.  
We consider it for X-Ray Free Electron Lasers (XFEL) \cite{huang2007review}, where methods for class averaging are not yet widely available.

The molecule is described by a 3D scalar field, $\phi(x,y,z)$.  
The $j$th 2D image is ideally modeled as a noisy and pixelated variant of:
\begin{equation} \label{eq:XFEL-Ij}
I_j := |\mathcal{P} \mathcal{F}(\phi \circ R_j) | : \mathbb{R}^2 \rightarrow \mathbb{R},
\end{equation}
where $R_j$ is an unknown rotation in $\mathbb{R}^3$, $\mathcal{F}$ is the Fourier transform, and $\mathcal{P}$ is the restriction to the plane $\hat{z}=0$.\footnote{This is disregarding the curvature of the Ewald sphere in XFEL.} 
The image records photon counts at each pixel, such that the counts are often assumed to be independent draws from a Poisson distribution having rate $I_j(\hat{x}, \hat{y})$.
In class averaging, we assume there is a limited set of ``significant views" of the molecule,  $\{\tilde{R}_j\}_{j=1}^r$.
The data consists of repeated noisy measurements of $I_j$ produced by these $\tilde{R}_j$. 
The goal is to recover $I_j$ from these measurements.\footnote{  
In this illustration, we neglect issues of matching images via in-plane rotations.
As ideally stated, the task is  mean estimation for a Poisson mixture. }

We simulate class averaging using the human ribosome molecule, which is 5LKS in the Protein Data Bank  \cite{myasnikov2016structure}.  
The scalar field $\phi$ is computed by the molmap function in Chimera \cite{chimera}, see \cref{fig:5lks} for a visualization.
Thirty important rotations $\tilde{R}_j$ are randomly chosen $(r=30)$.  
We compute $32 \times 32$ representations of the corresponding functions $I_j$ \eqref{eq:XFEL-Ij} ($n=1024$).
(Only the DFT coefficients of the nonnegative frequencies are kept, because of conjugate symmetry in \eqref{eq:XFEL-Ij}.)
We simulate a mixture of Poissons \eqref{eq:poisson} using these $32 \times 32$ coefficients as means vectors.
We pick the mixing weights by sampling a vector uniformly in $[1, 5]^n$ and normalizing. Finally, $p=20000$ XFEL images are simulated as samples from the Poisson mixture.

\begin{figure}[htbp]
    \centering
    \begin{minipage}{0.5\linewidth}
        \includegraphics[width = \linewidth]{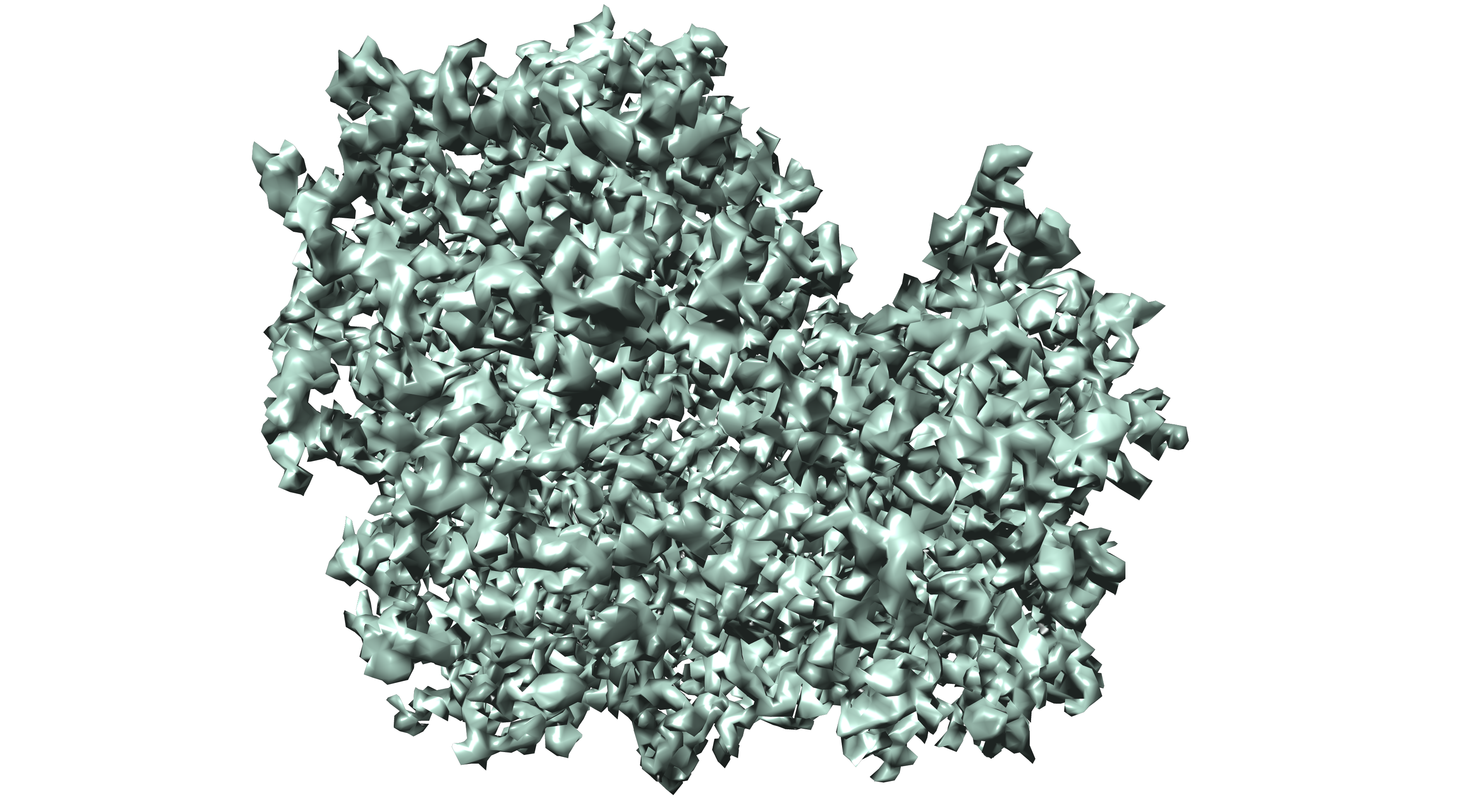}  
    \end{minipage}
    \quad
    \begin{minipage}{0.3\linewidth}
        \caption{The human ribosome: molecule 5LKS in the Protein Data Bank.}  
    \end{minipage}
    \label{fig:5lks}
\end{figure}

The \ALS algorithm takes $\sim 40$ minutes to converge.
It achieves an error of $0.9\%$ in weights and $0.5\%$ in means (as defined in \eqref{eq:error-metric}).
We benchmark it against \pom EM. 
Thirty $\textup{k-means}^{++}$ runs with a maximum of 10 iterations are carried out, with the best run used to initialize EM.
A single EM descent takes  $50 - 70$ minutes to converge. 
We repeat the EM procedure three times with different seeds. The best EM result still has an error of $> 13\%$ in the means.

The solutions are shown in \cref{fig:xfel}.
We show the ground-truth pictures, images computed by \ALS, and the best EM result.\footnote{ 
  For visualization purposes, the data is $\log$-transformed and we lower-bound the pixel values by $10^{-4}$.
  Respectively 1, 4, and 0 pixel values are changed for the ground-truth, \ALS and EM.}   
We circled the most obvious dissimilarities in the computed means and weights in red.
We remark that on real XFEL data, \ALS may enjoy extra advantages over EM --  
since \ALS is nonparametric, it should be more robust to violations of the Poisson noise model.

\begin{figure}[!htbp]
    \centering
    \subfloat{\includegraphics[width=0.8\linewidth]{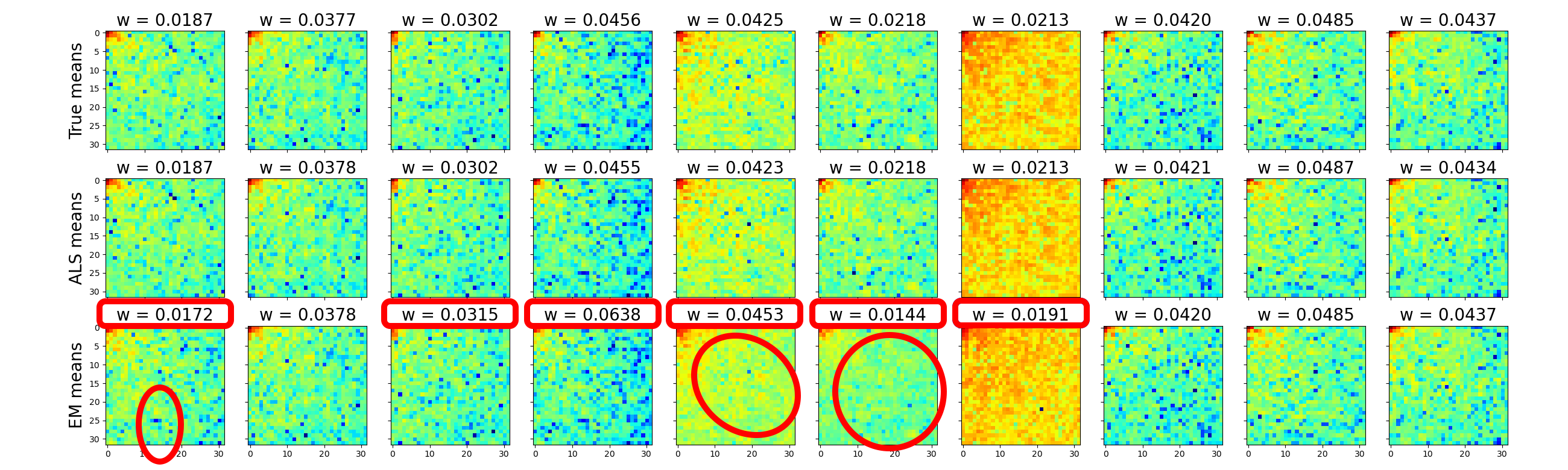}}\\
    \subfloat{\includegraphics[width=0.8\linewidth]{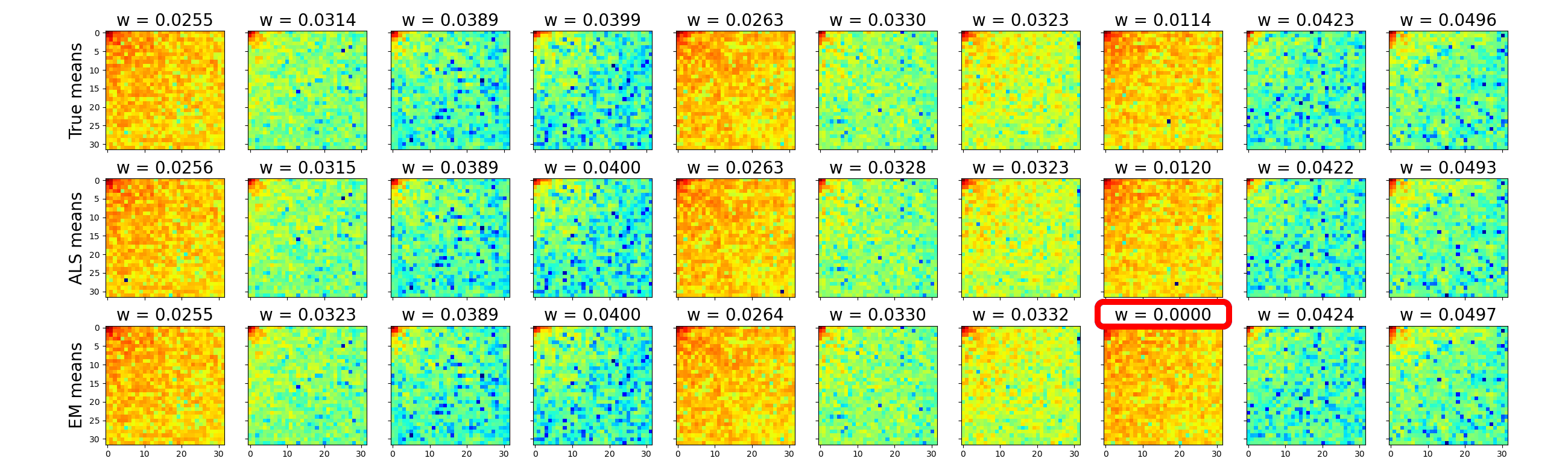}}\\
    \subfloat{\includegraphics[width=0.8\linewidth]{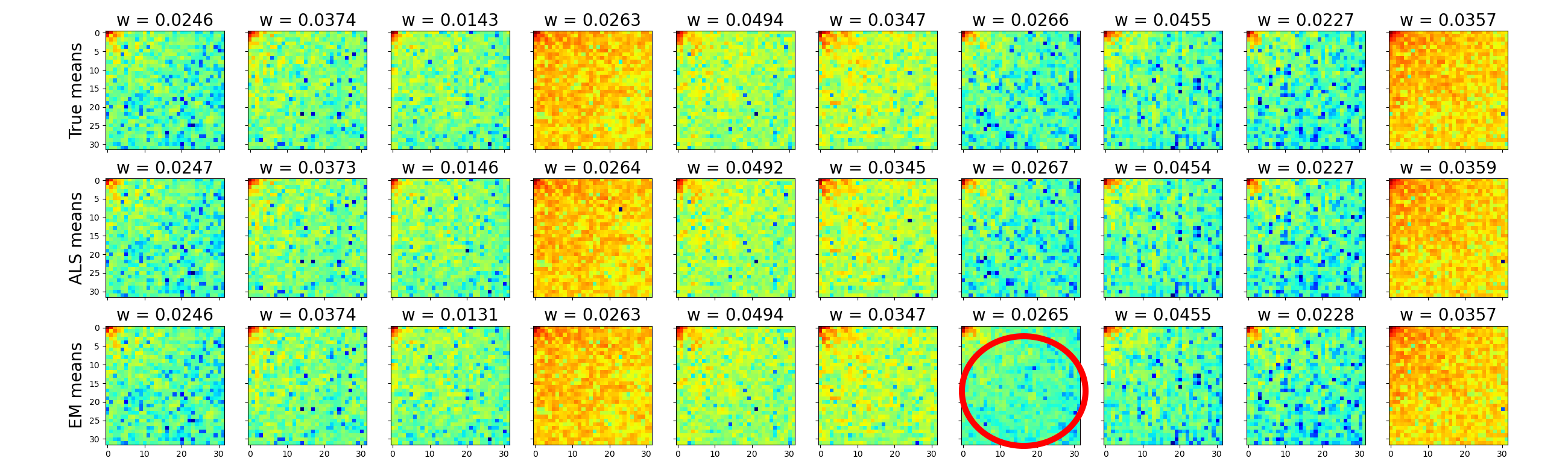}}  
    \caption{Results for the XFEL class averaging simulation, comparing \cref{alg:finalALS}  to \pom EM.  It is simulated as a Poisson mixture ($ p=20000, n=1024, r=30, d=4$) using images \eqref{eq:XFEL-Ij} of \cref{fig:5lks} as ground-truth means. 
 Top: ground-truth means; middle: computed means by \ALS\!\!; bottom: computed means by EM.}  
    \label{fig:xfel}
\end{figure}

\subsection{Application: Digit Recognition} 
\label{sec:usps}
We test the \ALS algorithm on a real dataset to show its performance when the conditional independence assumption does not hold.
We use the USPS dataset \cite{usps}, which is similar to the MNIST dataset \cite{deng2012mnist}, consisting of pictures of handwritten numbers and letters, though we only consider the digits.
Prior studies have fitted parametrized mixture distributions to the USPS dataset \cite{keysers2002combination,ma2009beta}.

Note that adjacent pixels are highly correlated.
We increase the number of clusters from $10$ to $26$, as each digit is written in different ways, and a larger number of product distributions may fit a mixture with correlated distributions better.
Each gray-scale image is normalized to have norm 1 and we attach to it its discrete cosine transform.
Even though attaching the DCT increases dependency between the coordinates, this produced the best visual results among different preprocessings.
We compare to the $\textup{k-means}^{++}$ solution from Python package \textit{scikit-learn}\footnote{\url{https://scikit-learn.org/stable/index.html}}, where 30 initializations are used and the best result is kept.

\begin{figure}[htbp]
    \centering
    \includegraphics[width = 0.8\linewidth]{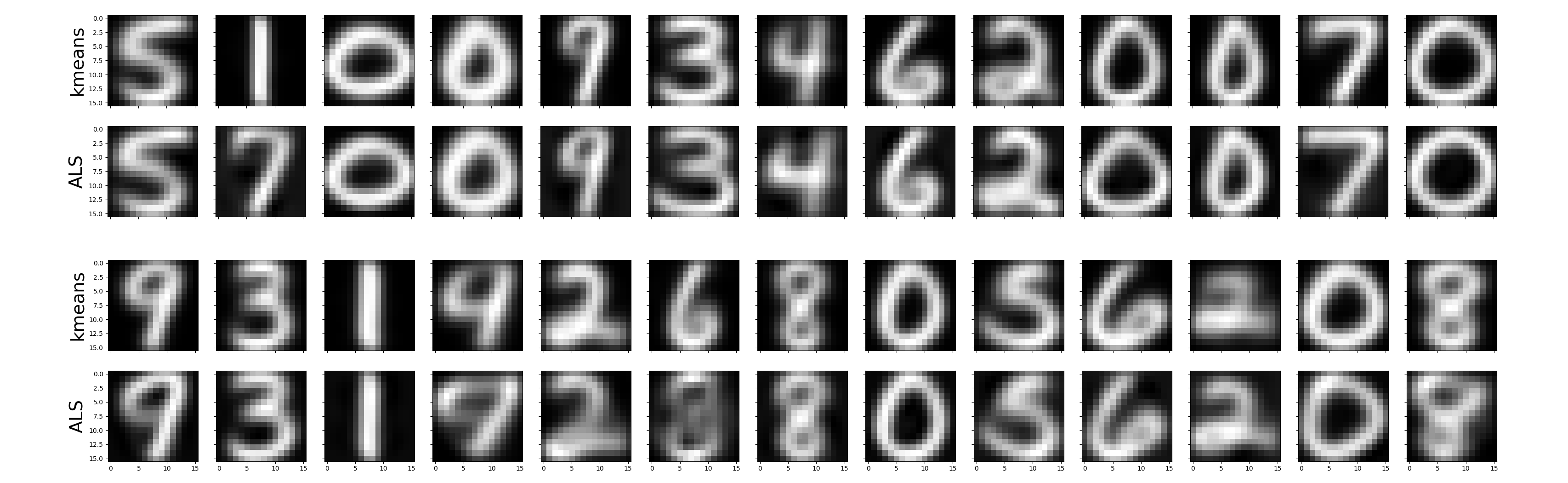}
    \caption{Results for the USPS dataset, comparing \cref{alg:finalALS} ($p=9296, n=512, r = 26, d=4$) to \textit{scikit-learn} $\textup{k-means}^{++}$\!.   
    Top: $\textup{k-means}^{++}$ solution; bottom: \ALS solution.}
    \label{fig:usps}
\end{figure}

The results are in \cref{fig:usps}. 
Images from \ALS and $\textup{k-means}^{++}$ are matched so that the Euclidean distance between them is minimized.
\ALS captures a different font of `7' in the second picture rather than duplicating another `1' like $\textup{k-means}^{++}$, but it errs on the first `8' in the second row.
Overall, the ALS framework does a reasonable job on this dataset even though entries are conditionally-dependent.

    \section{Discussion}
    This paper addressed the nonparametric estimation of conditionally independent mixture models \eqref{eq:cond-inp-mix}. 
    A two-step approach was proposed.
    First, we learned the mixing weights and component means from the off-diagonal parts of joint moment tensors.
    An efficient tensor-free ALS based algorithm was developed, building on kernel methods and recent works in implicit tensor decomposition. 
    Second, we used the learned weights and means to evaluate the general means functional $g\mapsto \E_{X\sim \cD_j}[g(X)]$ through a linear solve, again evading the formation of moment tensors.
    This procedure recovers the component distributions and higher moments.
    
    Numerical examples showed that the ALS framework has broad applicability.  It also exhibits potential advantages over popular alternative procedures, in terms of accuracy, reliability, and timing.
    On the theoretical side,
    we derived a system of incomplete tensor decompositions, relating the first few joint moments to mixing weights and component moments.
    Identifiability from these joint moments of the mixing weights and first few component moments was established under $r = \cO(n^{\floor{d/2}})$.
    Once the weights and means are known, the general means are shown to be identifiable from $\E[g(X) \otimes X^{\otimes {d-1}}]$ provided $r \leq \binom{n-1}{d-1}$.
    A local convergence analysis of the ALS algorithm was also conducted.  

    For the practicing data scientist, what is the significance? 
    First, we provided new tools for many mixture models which are well-suited to high dimensions.
    They are very
    competitive compared to popular packages, and use a completely different methodology.
    Second, we introduced a flexible nonparametric framework. 
    This is advantageous in real applications, especially when a good parameterization is not known a priori.

    We list possible future directions. 
    The MoM framework can be naturally applied to data streaming applications.
    The mean and weight estimation only depend on the moments of the data, so online compression tools can be applied to the moments.
    It should also be possible to generalize our methods to other dependency structures, e.g. conditionally sparsely-dependent mixtures.  
    Another direction is to conduct an application study.  The bio-imaging application in \cref{sec:xfel} is a prime target.
    
    \printbibliography
    
    \newpage
    \appendix

In these supplementary materials, \cref{apx:pf_mom} and \cref{apx:pf_conv} give the proofs to the theorems in the main text, respectively in \cref{sec:mom} and \cref{sec:alsconv}.
Then in \cref{apx:practice} we include full details of \ALS, briefly discussed in \cref{sec:improve}.
Derivative evaluation formulas are also given.
Finally in \cref{apx:hyper}, we list the hyperparameters used in the experiments in \cref{sec:exp}.

\section{Proofs for Identifiability}\label{apx:pf_mom}

\subsection{Proof of \cref{thm:low_rank_unique}} \label{apx:pf_low_rank_unique}

The main step is to uniquely fill in the diagonal entries of $\t T = \m P \big{(} \sum\nolimits_{j=1}^r \v a_j^{\otimes d} \big{)}$.
Then uniqueness of symmetric CP decompositions \cite{chiantini2017generic} implies the theorem.
Let $T_0$ be the reduced matrix flattening of $\t T$ (see \cref{lem:full_rk_flat_sym}).
As sketched in the main text, we locate $(r+1) \times (r+1)$ submatrices $\m M$ in $\m T_0$ in which only one entry $x$ is missing. 
\cref{lem:full_rk_flat_sym} allows us to compute the rank for submatrices of $\t T_0$. We give its proof now.

\begin{proof}(\textit{of \cref{lem:full_rk_flat_sym}})
Consider an $m \times m$ submatrix $\m M$ of $\m T_0$. 
Let its row and column labels be given by the words $\sR = \{\v i_1, \ldots, \v i_m\} \subset [n]^s$ and  $\sC = \{\v j_1, \ldots, \v j_m\} \subset [n]^t$ respectively.
Define masks $P_{\sR} : \mathbb{R}^{n^s} \rightarrow \mathbb{R}^m$ by $P_{\sR}(\t S) = (\t S_{\v i_1}, \ldots, \t S_{\v i_m})^{\top}$ and 
$ {P}_{\sC} : \mathbb{R}^{n^t} \rightarrow \mathbb{R}^m$ 
by ${P}_{\sC}(\t S) = (\t S_{\v j_1}, \ldots, \t S_{\v j_m})^{\top}$. 
We factor $\m M$ as follows:
\begin{equation}\label{eq:M-factor}
\m M \, = \, {\begin{pmatrix} P_{\sR}(\v a_1^{\odot s}) & \ldots & P_{\sR}(\v a_r^{\odot s}) \end{pmatrix}} {\begin{pmatrix} 
P_{\sC}(\v a_1^{\odot t}) & \ldots & P_{\sC}(\v a_1^{\odot t})
\end{pmatrix}^{\!\top}}
\end{equation}
where the factors are of size $m \times r$ and $r \times m$ respectively and $\odot$ denotes Khatri-Rao products.
In particular  $\operatorname{rank}(\m M) \leq r$, and without loss of generality we assume $m \leq r$. 
The Cauchy-Binet formula applied to \eqref{eq:M-factor} yields:
\begin{equation} \label{eq:cauchy-binet}
\det(\m M) = \!
 \sum_{1 \leq k_1 < \ldots < k_m \leq r} \!\!\!\!\! \det \! \begin{pmatrix} P_{\sR}(\v a_{k_1}^{\odot s}) & \cdots & P_{\sR}(\v a_{k_m}^{\odot s})\end{pmatrix} \det \!\begin{pmatrix} P_{\sC}(\v a_{k_1}^{\odot t}) & \cdots & P_{\sC}(\v a_{k_m}^{\odot t})\end{pmatrix}.
\end{equation}
We regard the entries of $\m A$ as independent variables.  
Each summand in \eqref{eq:cauchy-binet} is  a multihomogeneous polynomial, of degree $s+t$  in each of the columns $\v a_{k_1}, \ldots, \v a_{k_m}$.
Furthermore, each is a nonzero polynomial because
\begin{equation*}
\det \! \begin{pmatrix} P_{\sR}(\v a_{k_1}^{\odot s}) & \cdots & P_{\sR}(\v a_{k_m}^{\odot s})\end{pmatrix}  = \sum_{\pi \in S_{m}} \v a_{k_{\pi(1)}}^{\v i_1} \! \ldots \v a_{k_{\pi(m)}}^{\v i_m}
\end{equation*}
consists of distinct monomials (as the $\v i$'s are distinct), and likewise for the determinant involving the $P_{\sC}$'s in \eqref{eq:cauchy-binet}, where we denoted $\v a^{\v i} = a_1^{i_1} \ldots a_{n}^{i_n}$.
Therefore $\det(\m M)$ is a nonzero polynomial.  So for generic values of $\v A$ it gives a nonzero number, and $\operatorname{rank}(\m M) = m$.    The proof of \cref{lem:full_rk_flat_sym} is complete.  
\end{proof}

By \cref{lem:full_rk_flat_sym}, $\m M \in \R^{(r+1)\times (r+1)}$ is rank $r$ and the minor of the only missing entry $x$ in $\m M$ is nonzero, so we can uniquely recover it, as stated below.  

\begin{corollary}\label{lem:fillmat}
   In the setting of \cref{lem:full_rk_flat_sym}, assume $r < \binom{n+s-1}{s}$ and that $\m A$ is Zariski-generic.  Let $\m M$ be an $(r + 1) \times (r + 1)$ submatrix of $\m T_0$. 
   Then any entry $m_{ij}$ of $\m M$ is uniquely determined given  all other entries in $\m M$.
\end{corollary}

\begin{proof}
By \cref{lem:full_rk_flat_sym}, 
 both $\m M$ and $\m M'$ have rank $r$, where $\m M' \in \mathbb{R}^{r \times r}$ is obtained by dropping row $i$ and column $j$ from $\m M$.  
So $\det(\m M) = 0$ and $\det(\m M') \neq 0$. 
However by Laplace expansion, $\det(\m M) = (-1)^{i+j} \det(\m M') m_{ij} + b$, where $b$ only depends on entries of $\m M$ besides $m_{ij}$.  The result follows.
\end{proof}

Finally, we show the details of filling the diagonal entries in $\t T$ and prove \cref{thm:low_rank_unique}.

\begin{proof} (\textit{of \cref{thm:low_rank_unique}})
    It suffices to show  the diagonal entries in $\t T$ are uniquely determined by  $ r \leq \binom{\lfloor (n-1)/2 \rfloor}{\lfloor d/2 \rfloor}$ and knowledge of $\m P \t T$.
    For then by \cite[Thm.~1.1]{chiantini2017generic}, the CP decomposition of $\t T$ is unique implying $\m A = \m B \m \Pi$, except for the seven weakly defective cases in \cite{chiantini2017generic}.
    But one can check that under our rank bound, all weakly defective cases with $d < n$ are avoided.

    For the diagonal completion we start by flattening $\t T$ to its reduced matrix flattening $\m T_0$ where $s = \lfloor d/2 \rfloor$ and $t = \lceil d/2 \rceil$ (see \cref{lem:full_rk_flat_sym}). 
    Under our rank bound,
    \begin{equation*}
        \rank(\m T_0) = \min \Big{\{}r, \binom{n + \floor{d/2} - 1}{\floor{d/2}}\Big{\}} = r.
    \end{equation*}
    For an index $(\v i, \v j) = (i_1\ldots i_{\floor{d/2}},~j_1\ldots j_{\ceil{d/2}})$ and  letter $b\in [n]$, define its $\v i$ and $\v j$ \defil{multiplicity} to respectively be the number of instances of $b$ in  $\v i$ and $\v j$.
    Define its $(\v i, \v j)$ multiplicity to be the sum of its $\v i$ and $\v j$ multiplicity.
    By \defil{degree} of $\v i$, $\v j$, and $(\v i, \v j)$, we shall respectively mean the maximal multiplicity of any letter in $\v i$, $\v j$, and $(\v i, \v j)$.

    We follow the completion strategy indicated by the path \eqref{eq:fillpath}, by first  solving for the diagonal entries with $\v i$ and $\v j$ having degree 1  by an  induction on the number $k$ of common letters  between $\v i$ and $\v j$.
    In the base case $k = 1$, 
    without loss of generality, we may take $i_1 = j_1 = 1$ (it is not hard to repeat the following argument for any $i_a = j_b = c$).
    Now we split the set $[n]\setminus\set{1}$ into two parts $R$ and $C$ evenly (i.e. $|R| \leq |C| \leq |R| + 1$), such that $\set{i_2, \ldots, i_{\floor{d/2}}} \cap C = \emptyset$, and $\set{j_2, \ldots, j_{\ceil{d/2}}}\cap R = \emptyset$. This is possible whenever $d < n$.
    Let the collection of all non-decreasing words of $C$ of size $\ceil{d/2}$ be $C_s$, and
    the collection of all strictly increasing words of $R$ of size $\floor{d/2}$ be $R_s$.
    By the low rank assumption, $|R_s| = \binom{\floor{(n-1)/2}}{\floor{d/2}} \geq r$, $|C_s| = \binom{\ceil{(n-1)/2}}{\ceil{d/2}} \geq r$.
    Consider the submatrix of $\m T$ with rows indexed by $\set{\v i} \cup R_s$ and columns indexed by $\set{\v j} \cup C_s$.
    $(\v i, \v j)$ is the only missing entry in this submatrix of size $(1 + |R_s|)\times(1 + |C_s|) \geq (1 + r) \times (1 + r)$. 
    Therefore by \cref{lem:fillmat}, this entry is uniquely determined.

    Next, with the base case proved, suppose we have recovered all such entries with at most $k-1$ common letters.
    Without loss of generality, let us solve for entry $\v i = 12...ki_{k+1}\ldots i_{\floor{d/2}}$ and $\v j = 12...kj_{k+1}\ldots j_{\ceil{d/2}}$ with the $i$'s and $j$'s all different and greater than $k$.
    Again, let $\hat{C}$ and $\hat{R}$ be an even split of the set $[n]\setminus[k]$ (i.e. $|\hat{R}| \leq |\hat{C}| \leq |\hat{R}| + 1$), such that $\set{i_{k+1}, \ldots, i_{\floor{d/2}}} \cap C = \emptyset$, and $\set{j_{k+1}, \ldots, j_{\ceil{d/2}}}\cap R = \emptyset$.
    Then define $C = \set{2,\ldots,k}\cup \hat{C}$, $R = \set{2, \ldots, k}\cup \hat{R}$. Then $|C| = \ceil{(n-1)/2}$ and $|R| = \floor{(n-1)/2}$.
    Define collections $C_s$ and $R_s$ as before, and consider the submatrix of $\m T$ with rows indexed by $\set{\v i} \cup R_s$ and columns indexed by $\set{\v j} \cup C_s$. 
    This is again a matrix of shape $(1 + |R_s|)\times(1 + |C_s|) \geq (1 + r) \times (1 + r)$, and the only missing entry is $(\v i, \v j)$ since all other entries have  $\leq k-1$ common letters between $\v i$ and $\v j$.

    Finally, we do another induction step on the degree. 
    Suppose the entries with $(\v i, \v j)$ degree at most $2^{k}$ are solved, we will recover those with degree at most $2^{k+1}$. 
    The base case of degree 2 has been proved, since by symmetry of the tensor, for any index $(\v i, \v j)$ with degree 2, we can rearrange the indices so that the $\v i$ degree and $\v j$ degree are both 1, and these entries have been recovered in the last step.

    Now suppose we are done with $(\v i, \v j)$ whose degree is at most $2^{k}$ ($k \geq 1$). 
    For index $(\v i, \v j)$ with degree at most $2^{k+1}$, by symmetry of the tensor, we can rearrange the indices such that for each letter,
    its $\v i$ and $\v j$ multiplicity differ by at most 1, and thus are both at most $2^{k}$. 
    Without loss of generality, let $1, \ldots, m$ be those  letters with $(\v i, \v j)$ multiplicity at least $2^{k} + 1$. 
    For all other numbers in the index, their $\v i$ and $\v j$ multiplicity are both at most $2^{k-1}$.
    Now let $C = R = [n]\setminus[m]$, then since $3 m \leq (2^k + 1)m \leq d$, we have $|C| = |R| \geq n - \floor{d/3} \geq \ceil{(n-1)/2}$. 
    Thus, with $R_s$ and $C_s$ defined as above, $|R_s| = \binom{|R|}{\floor{d/2}} \geq r$, and $|C_s| = \binom{|C|}{\ceil{d/2}} \geq r$.
    Again, consider the submatrix of $\m T$ with rows indexed by $\set{\v i} \cup R_s$ and columns indexed by $\set{\v j} \cup C_s$. 
    This matrix of shape $(1 + |R_s|)\times(1 + |C_s|) \geq (1 + r) \times (1 + r)$ has the only missing entry $(\v i, \v j)$, since all other entries have an $(\v i, \v j)$ degree of at most $1 + 2^{k-1} \leq 2^k$.

    With this established, we finish the completion procedure.  
    As explained in the first paragraph, \cref{thm:low_rank_unique}  follows.
\end{proof}

\subsection{Proof of \cref{thm:gmeans_id}} \label{apx:pf_gmeansid}

\begin{proof}(\textit{of \cref{thm:gmeans_id}})
    We solve $\m Y = (\v y_1|\ldots|\v y_r)$ row-by-row where $\v y_j = \E_{X\sim\cD_j}[g(X)]$.
    By \cref{prop:general-equations},
    \begin{equation*}
        \m P\Big{(} \sum\nolimits_{j=1}^r w_j \v y_j \otimes (\v m_j^1)^{\otimes d-1}\Big{)}
        =
        \m P \Big{(} \mathbb{E}_{X \sim \cD} [g(X) \otimes X^{\otimes d-1}]\Big{)}.
    \end{equation*}
    We solve the $k$th row $\v y^k$ by setting the $\ell^2$ distance between the two sides to $0$.
    Applying the projection $\sym()$ on both sides and following the same variable separation approach in \eqref{eq:separation}, we get
    \begin{equation*}
        d\gt_d \, \Big{\|}\P \Big(
            \sum\nolimits_{j = 1}^r (\v \gb^k)_j (\v a_j^{(k)})^{\otimes i-1} 
            - 
            \sym \E_{X\sim \cD}\big[g_k(X_k) \cdot (X^{(k)})^{\otimes d-1}\big] 
        \Big{)}\Big{\|}^2,
    \end{equation*}
    where $\v \gb^k = \v w * \v y^k$.
    This is exactly the least squares problem solved by \cref{alg:gmom}.
    Since $r \leq \binom{n-1}{d-1}$ and $\gt_d > 0$, according to \cref{prop:full-rank-general}, this linear least squares problem has a unique minimizer.
\end{proof}

\section{Proofs for ALS} \label{apx:pf_als}

\subsection{Proofs for Invertibility of $\m L$}\label{apx:pf_invL}

We start with a lemma that is helpful for proving the invertibility of $\m L$.

\begin{lemma} \label{lem:offdiag_indep}
    Let $\set{\v u_j}_{j = 1}^r \subset \R^n$ be Zariski-generic.
    Then $\big{\{}\P(\v u_j^{\otimes d})\big{\}}_{j = 1}^r$ is linearly independent if and only if
    $r \leq \binom{n}{d}$.
\end{lemma}

\begin{proof}
    We form an $\binom{n}{d} \times r$ matrix $\m M$,
    with rows indexed by the size $d$ subsets of $[n]$, such that $\operatorname{rank}(\m M) = \operatorname{dim} \operatorname{span} \big{(}\P(\v u_j^{\otimes d}) : j =1, \ldots, r \big{)} $ as follows:
    \begin{equation*}
        \m M_{\{i_1,\ldots,i_{d}\}, j} = (\v u_j)_{i_1} (\v u_j)_{i_2}\cdots (\v u_j)_{i_{d}}.
    \end{equation*}
    If $r \leq \binom{n}{d}$, $\m M$ is tall and thin, and it is straightforward to verify that the determinantal polynomial for any $r \times r$ submatrix is nonzero as in the proof of \cref{lem:full_rk_flat_sym}.
    If $r > \binom{n}{d}$, then the columns cannot be all linearly independent.  
\end{proof}

The proofs for invertibility of $\m L$ for the weight and mean updates are similar.
We prove then respectively below.

\begin{proof}(\textit{of \cref{prop:invertL_w}})
    The cost function at order $d$ is given by
    \begin{equation*}
        f^{(d)} = \gt_d \Big{\|}\P \big{(}\sum_{j = 1}^r w_j {\v a_j}^{\otimes d} - \t T\big{)}\Big{\|}^2.
    \end{equation*}
    Here $\t T$ is a constant tensor independent of $\v w$.
    By \cref{lem:offdiag_indep}, the columns $\P(\tprod{\v a_j}{d})$ of the coefficient matrix of the least squares problem are generically linearly independent given $r \leq \binom{n}{d}$. Since $\gt_d > 0$, there is a unique minimizer $\v w$ of $f^{(d)}$.
    Adding costs $f^{(i)}$ of lower orders is appending more rows to the coefficient matrix of the least squares problem, and thus it can only increase the (column) rank of the coefficient matrix.
    Therefore, we conclude that the matrix $\m L$ in the normal equation is invertible.
\end{proof}

\begin{proof}(\textit{of \cref{prop:invertL_a}})
    Assume row $k$ is being updated.
    Consider the cost at order $d$.
    We solve $\v \gb^k$ that minimizes (see \eqref{eq:alsd-1})
    \begin{equation*}
        d\gt_d \Big{\|}\P \big{(}\sum_{j = 1}^r (\v \gb^k)_j {\v a_j^{(k)}}{^{\otimes d-1}} - \t T\big{)}\Big{\|}^2,
    \end{equation*}
    where $\v a_j^{(k)}$ removes its $k$th entry, and $\t T$ is a constant tensor independent of $\v \gb^k$.
    Similar to the proof of \cref{prop:invertL_w}, the coefficient matrix for this least squares problem is generically linearly independent given $r \leq \binom{n-1}{d-1}$. Since $\gt_d > 0$, there is a unique minimizer $\v \gb^k$ of $f^{(d)}$.
    Again adding lower order costs will only increase the column rank.
    Therefore, the matrix $\m L$ in the normal equation is invertible.
\end{proof}

\subsection{Proofs for Convergence}\label{apx:pf_conv}
We start by stating the main tool 
for local linear convergence coming from the alternating optimization literature.

\begin{theorem} (See \cite[Thm.~2, Lem.~1-2]{bezdek2003convergence}) \label{lem:alsconv}
    Consider a minimization problem 
    \begin{equation}
        \min_{\v x_1,\ldots \v x_n} f(\v x_1,  \ldots,\v x_n),
    \end{equation}
    and the alternating optimization (AO) scheme:
    \begin{equation*}
        \begin{cases}
            \v x_1^{(t+1)} & = \argmin_{\v x_1} f(\v x_1, \v x_2^{(t)}, \ldots, \v x_n^{(t)}) \\[0.1em]
            \v x_2^{(t+1)} & = \argmin_{\v x_2} f(\v x_1^{(t+1)}, \v x_2, \ldots, \v x_n^{(t)}) \\[0.1em]
            & \,\, \vdots \\[0.1em]

            \v x_n^{(t+1)} & = \argmin_{\v x_n} f(\v x_1^{(t+1)}, \ldots, \v x_{n-1}^{(t+1)}, \v x_n)
        \end{cases}
    \end{equation*}
    Let $\v x^*$ be a local minimum.
    If $f$ is $C^2$ in a neighborhood around $\v x^*$ and the Hessian $\m H$ of $f$ at $\v x^*$ is positive-definite,
    then the spectral radius satisfies 
    \begin{equation*}
    \gr\left(\m I - \m M^{-1} \m H(\v x^*)\right) < 1,
    \end{equation*}
    where $\m M$ is the  lower block-diagonal part of $\m H$ with respect to the ordering $\v x = (\v x_1,\ldots,\v x_n)$.
    Define the errors $\v e^{(t)} = \v x^{(t)} - \v x^*$. If $f$ is $C^3$, then 
    \begin{equation} \label{eq:errit}
        \v e^{(t+1)} = \left(\m I - \m M^{-1} \m H(\v x^*)\right)\v e^{(t)} + \cO(\|\v e^{(t)}\|^2),
    \end{equation}
    and so AO locally converges linearly to $\v x^*$.
\end{theorem}

In our problem the cost function $f = f^{[d]}$ \eqref{eq:costfn_gen} is a polynomial of degree $2d$, so  has enough regularity.
The main thing we need to verify is the positive definiteness of the Hessian.
We give two preparatory lemmas before proving Theorem~\ref{thm:alsconv}.
\begin{lemma} \label{lem:positiveH}
    Let $f$ be a polynomial map between Euclidean spaces. Consider the problem
    \begin{equation} \label{eq:optprob}
        \min_{\v x} \, g(\v x; \v y) := \frac{1}{2}\norm{f(\v x) - \v y}^2.
    \end{equation}
    Suppose that $f$ is generically finite-to-one over the complex field, i.e. $\{\text{complex } {\v z} : f({\v z}) = f(\v x) \}$ are finite sets for Zariski-generic complex $\v x$.
    Then the Hessian over the reals $\m H_g(\v x; f(\v x)) \succ 0$ for Zariski-generic real $\v x$.
\end{lemma}

\begin{proof}
    Denote $\m J_{f}$, $\m J_{g}$, $\t H_{f}$, $\m H_{g}$ the gradients and Hessians of $f$ and $g$ with respect to $\v x$ respectively.
    Note that $\t H_f$ is a third-order tensor.
    Direct computation shows 
    \begin{gather}
        \v J_g(\v x; \v y) = \v J_f(\v x)^\top (f(\v x) - \v y), \label{eq:Jg}\\
        \m H_g(\v x; \v y) = \v J_f(\v x)^\top \v J_f(\v x) + \t H_f(\v x) \cdot (f(\v x) - \v y). \label{eq:Hg}
    \end{gather}
    Since $f$ is generically finite-to-one, $\v J_f$ is full column rank at generic points by polynomiality \cite{sommese2005numerical}.
    At $\v y = f(\v x)$, the residual $f(\v x) - \v y$ is zero, thus $\m H_g(\v x; f(\v x)) \succ 0$.
\end{proof}

\begin{lemma} \label{lem:ctglbmin}
    Under the setup of \cref{lem:positiveH}, assume in addition that $f$ is generically one-to-one over the complex field, i.e. 
    $\{\text{complex } {\v z} : f({\v z}) = f(\v x) \} = \{ \v x \}$ for Zariski-generic complex $\v x$.
    Then for Zariski-generic real $\v x^*$,
    there exists a real Euclidean neighborhood $U$ of $\v y^* = f(\v x^*)$ such that the optimization problem \eqref{eq:optprob} has a unique solution for all $\v y \in U$, 
    and the function $h$ mapping any $\v y \in U$ to the unique solution is smooth.
\end{lemma}
\begin{proof}
    For generic $\v x^*$ there is a Euclidean neighborhood $E$ of $\v x^*$ such that
    $\cM = f(E)$ is a smooth manifold; $\{ {\v x} : f({\v x}) = \v y \}$ is a singleton for all $\v y \in \mathcal{M}$; and $f^{-1} : f(E) \rightarrow E$ is smooth (see e.g.  \cite[Thm. A.4.14]{sommese2005numerical}).
    Then \cite[Lem. 3.1]{absil2012projection} implies there is a uniquely-defined  projection $P_{\cM}$ (with respect to $\ell_2$) mapping a Euclidean neighborhood $U$ of $\v y^*$ into $\cM$, which is smooth.
    The desired uniqueness now follows, and 
    $h = f^{-1} \circ P_{\cM}$ is well-defined and smooth from $U$ to $E$ mapping $\v y$ in $U$ to the unique minimizer.
\end{proof}

\begin{proof}(\textit{of \cref{thm:alsconv}}).
    Firstly, since we have an ALS algorithm, the cost function is non-increasing after each update.
    Now we prove the assertions 1 to 3.

    For assertion 1, the goal is to apply \cref{lem:alsconv} to \cref{alg:basicALS} with the weight constraints imposed.
    Notice that to prove local convergence, we may ignore the inequality constraint $\v w > \v 0$ and only work with the sum-to-one constraint.
    Indeed, if local linear convergence holds under the sum-to-one constraint,
    we can start at $\v w_0 > \v 0$ close enough to $\v w_c$ such that the inequality constraint $\v w > \v 0$ is never violated under linear convergence (since $\v w_c$ is in the interior of $\gD_{r-1}$).
    This then establishes local linear convergence with the inequality constraint.
    To deal with the sum-to-one constraint, let $\v w^\circ = (w_1,\ldots,w_{r-1}) \in \mathbb{R}^{r-1}$ be a representation of the weights $\v w\in \gD_{r-1}$ using $w_r = 1 - \sum_{i < r}w_i$.
    Then positive definiteness of the Hessian at $(\v w_c, \m A_c)$ implies the Hessian with respect to $(\v w_c^\circ, \m A_c)$ is also positive definite.
    Therefore, by \cref{lem:alsconv}, local linear convergence follows.

    Now we prove assertion 2.
    Since $\v w^\circ$ is an one-to-one representation of $\v w$,
    according to \cref{thm:wMid}
    the map $f: (\v w^\circ, \m A) \mapsto (\P(\sum_j w_j \v a_j),\ldots,\P(\sum_j w_j \tprod{\v a_j}{d}))$ is still a generically one-to-one polynomial map, and there is a unique preimage $(\v {w^*}^\circ, \m A^*)$ of the exact moments $(\P(\t M^1),\ldots,\P(\t M^d))$.
    Define
    \begin{equation*}
        \tilde{f}^{[d]}(\v w^\circ, \m A ; \set{\t T^i}) := \sum\nolimits_{i = 1}^d \gt_i \big{\|} \P\big{(}\sum\nolimits_{j = 1}^r w_j \v a_j^{\otimes i} - \t T^i\big{)}\big{\|}^2,
    \end{equation*}
    Then by \cref{lem:positiveH}, the Hessian of $\tilde{f}^{[d]}$ with respect to $(\v w^\circ, \m A)$ is positive definite at the exact solution $({\v w^*}^\circ, \m A^*; \set{\t M^i})$.
    By \cref{lem:ctglbmin}, there are open neighborhoods $\cO_i$ for each $\t M^i$, so that when $\t M^i$ is replaced with $\widehat{\t M}_p^i$ within $\cO_i$, there is again a unique solution $(\v w_p^\circ, \m A_p)$ that minimizes $\tilde{f}^{[d]}$. This proves the first part of assertion 2.
    By \cref{lem:ctglbmin}, the function
    \begin{equation*}
        h: (\widehat{\t M}_p^1,\ldots, \widehat{\t M}_p^d)\mapsto (\v w_p^\circ, \m A_p)
    \end{equation*}
    is smooth around the exact moments.
    Note that $({\v w^*}^\circ, \m A^*) = h(\t M^1,\ldots, \t M^d)$ and at the exact solution  $({\v w^*}^\circ, \m A^*, \set{\t M^i})$ the Hessian of $\tilde{f}^{[d]}$ is shown to be positive definite.
    Since the Hessian depends continuously on the inputs $(\v w, \m A, \set{\t T^i})$, it follows that we can choose $\cO_i$ appropriately so that the Hessian is positive definite at $(\v w_p, \m A_p)$.
    This completes the proof for assertion 2.
    
    Finally by the law of large number $\widehat{\t M}_p^i\rightarrow\t M^i$ almost surely, and thus by continuity of $h$ we have $(\v w_p, \m A_p) \rightarrow (\v w^*, \m A^*)$ almost surely, establishing assertion 3.
\end{proof}

Next we quantify the convergence rates for the ALS algorithm.

\begin{proof}(\textit{of \cref{prop:linconv}})
    Let $\v e$ be the error in the weights and $\m E$ be the error in the mean matrix.
    As $\ind{}^{\top}\v e = 0$, we can write $\v e = \m Q\v \gee$ where the columns of $\m Q \in \R^{r\times(r-1)}$ are an orthonormal basis for the subspace orthogonal to $\ind{}$. 

    We define the reparametrized function
    \begin{equation*}
        \theta(\v \gee, \m E) := f^{[d]}(\v w^*_p + \m Q\v \gee, \m A^*_p + \m E; \m V_p).
    \end{equation*}
    Since $f^{[d]}$ is a polynomial with sufficient regularity, so is $\theta$. 
    By \cref{lem:alsconv}, the error iterates by applying a matrix  characterized by the Hessian of $\theta$ at $0$.
    We augment $\m Q$ to $\widehat{\m Q} = \Diag{\m Q, \m I_{nr}}$,
    and have
    \begin{equation*}
        \m H_{\theta}(0, 0) = \widehat{\m Q}^\top \m H_{f^{[d]}}(\v w^*_p, \m A^*_p; \m V_p) \widehat{\m Q} = \m H.
    \end{equation*}
    Thus by \cref{lem:alsconv}, the convergence rate comes from the spectral radius
    \begin{equation*}
        \gr(\m I - \m M^{-1} \m H) < 1,
    \end{equation*}
    where $\m M$ is the block lower triangular part (partitioned by $(\v w, \v a^1,\ldots,\v a^n))$ of $\m H$.
    The formula for $\m H_{f^{[d]}}$ is shown after the proof of \cref{prop:pconv} below.
\end{proof}

\begin{proof} (\textit{of \cref{prop:pconv}})
    We rewrite least squares problem as an instance of \eqref{eq:optprob} by defining   
    \begin{equation*}
       f(\v w, \m A) = \begin{pmatrix}
            \sqrt{\gt_1} \, T^1(\v w, \m A)\\
            \vdots\\
            \sqrt{\gt_d} \, T^d(\v w, \m A)
        \end{pmatrix},
        \ \ \ \ \ \
        \v y^* = \begin{pmatrix}
            \sqrt{\gt_1} \, \P(\t M^1)\\
            \vdots\\
            \sqrt{\gt_d} \, \P(\t M^d)
        \end{pmatrix}.
    \end{equation*}
    Now by the identifiability, the optimization problem
    \begin{equation*}
        \argmin_{\v w, \m A} g(\v w, \m A; \v y^*) 
        =
        \argmin_{\v w, \m A} \norm{f(\v w, \m A) - \v y^*}^2
    \end{equation*}
    admits a unique minimizer $(\v w^*, \m A^*)$.
    Denote
    \begin{equation*}
        \v y_p = 
        \begin{pmatrix}
            \sqrt{\gt_1} \P(\widehat{\t M}_p^1)
            \\
            \vdots\\
            \sqrt{\gt_d} \P(\widehat{\t M}_p^d)
        \end{pmatrix}.
    \end{equation*}
    As $p\rightarrow\infty$, we have $\v y_p \rightarrow \v y^*$ almost surely, and by \cref{lem:ctglbmin},
    for all $\gee_p$ sufficiently small there is a smooth function $h$ mapping $\v y_p$ to the unique global minimizer $(\v w_p, \m A_p)$.
    With $\gee_p = \| \v y_p - \v y^*\|$,

    \begin{equation*}
        \norm{h(\v y_p) - h(\v y^*)} 
        \leq
        \norm{\m J_h(\v y^*)}_{\text{op}} \cdot \gee_p + C \gee_p^2,
    \end{equation*}
    where $\m J_h$ is the Jacobian of $h$ and $C$ is a constant depending on the Hessian of $h$ at $\v y^*$. 
    Since the Hessian $\partial^2_{(\v w, \m A)} g(\v w, \m A; \v y) \succ 0$ at $(\v w^*, \m A^*, \v y^*)$, by the implicit function theorem the function $h$ is identified with the implicitly-defined function mapping $\v y$ around $\v y^*$ to $(\v w, \m A)$ such that $\partial_{(\v w, \m A)} g(\v w, \m A; \v y) = 0$.
    Given this we compute the Jacobian of $h$ using equations \eqref{eq:Jg} and \eqref{eq:Hg}:
    \begin{align*}
        \m J_h(\v y^*)
        &=
        -\left[ \partial_{(\v w, \m A)}^2 g(\v w, \m A; \v y) \right]^{-1} 
        \left(\partial_{\v y}\partial_{(\v w, \m A)}g(\v w, \m A; \v y)\right)
        \bigg\rvert_{\v w^*, \m A^*, \v y^*} 
        \\
        &=
        \left[\m J_f(\v w^*, \m A^*)^\top \m J_f(\v w^*, \m A^*)\right]^{-1}
        \m J_f(\v w^*, \m A^*)^\top,
    \end{align*}
    which is the pseudo-inverse of $\m J_f(\v w^*, \m A^*)$.
    Recall that $\m J_i$ is the Jacobian of $T^i$ at $(\v w^*, \m A^*)$. By definition of $f$,
    \begin{equation*}
        \m J_f(\v w^*, \m A^*)^\top \m J_f(\v w^*, \m A^*)
        =
        \sum_{i = 1}^d \gt_i \m J_i^\top \m J_i.
    \end{equation*}
    Since generically $\m J_f$ is full rank, $\gs = \gs_{\min}\left(\sum_{i = 1}^d \gt_i \m J_i^\top \m J_i\right) > 0$ and $\norm{\m J_h(\v y^*)}_{\text{op}} = \gs^{-1/2}$. This completes the proof.
    Together with the formulas for $\m H_{f^{[d]}}$, we give the formulas for $\m J_i^\top \m J_i$ after this proof.
\end{proof}

Here are the formulas mentioned in the last sentences of \cref{prop:linconv} and \cref{prop:pconv}.
Since $f^{[d]} = \sum_{i = 1}^d \gt_i f^{(i)}$, it suffices to list formulas for $\m H = \m H_{f^{(d)}}$ and $\m M = \m J_d^\top \m J_d$.
We denote $\v e_i$ the $i$th standard basis vector and denote the entries
\begin{gather*}
    M_{ij, kl} = \ip{\pd{T^d}{(\v a_i)_j}}{\pd{T^d}{(\v a_k)_l}},\ \ \ 
    M_{i, jk} = \ip{\pd{T^d}{w_i}}{\pd{T^d}{(\v a_j)_k}},\ \ \ 
    M_{i, j} = \ip{\pd{T^d}{w_i}}{\pd{T^d}{w_j}},\\\
    H_{ij, kl} = \frac{\partial^2 f^{(d)}}{\partial (\v a_i)_j \partial (\v a_k)_l},\ \ \ 
    H_{i, jk} = \frac{\partial^2 f^{(d)}}{\partial w_i \partial (\v a_j)_k},\ \ \ 
    H_{i, j} = \frac{\partial^2 f^{(d)}}{\partial w_i \partial w_j}.
\end{gather*}
We give formulas for these entries evaluated at general $(\v w, \m A; \m V)$, where \text{sym} denotes the symmetrization operator:
\begin{gather*}
    M_{ij, kl} = \left\langle d w_i \P \, \text{sym} (\v a_i^{\otimes d-1}\otimes \v e_j),~ d w_k \P \, \text{sym} (\v a_k^{\otimes d-1}\otimes \v e_l) \right\rangle;\\
    M_{i, jk} = \left\langle \P \left(\v a_i^{\otimes d}\right),~ d w_k \P \, \text{sym} (\v a_k^{\otimes d-1}\otimes \v e_l) \right\rangle;\\
    M_{i, j} = \left\langle \P \, (\v a_i^{\otimes d}),~ \P \,  (\v a_j^{\otimes d}) \right\rangle.
\end{gather*}
Now denote the residual tensor as $\t R = \P\big{(}\sum\nolimits_{j=1}^r w_i \v a_i^{\otimes d} - \frac{1}{p}\sum\nolimits_{\ell=1}^p \v v_{\ell}^{\otimes d}\big{)}$. Then
\begin{gather*}
    H_{ij, kl} = M_{ij, kl} + \delta_{ik} w_k d(d-1) \left\langle \v a_i^{\otimes d-2}\otimes \v e_j \otimes \v e_l,~\t R \right\rangle;\\
    H_{i, jk} = M_{i, jk} + \delta_{ij} d \left\langle \v a_j^{\otimes d-1} \otimes \v e_k,~ \t R \right\rangle;\\
    H_{i, j} = M_{i, j}.
\end{gather*}

\section{Further Algorithm Details for \cref{sec:improve}}\label{apx:practice}

\subsection{Anderson Acceleration} \label{apx:anderson}

The baseline ALS \cref{alg:basicALS} speeds up if we use Anderson Acceleration (AA) \cite{walker2011anderson} during its terminal convergence.
This works similarly to the application of AA to the CP decomposition in \cite{sterck2012nonlinear}.  

We start AA by applying one ALS update to the current iterate $({\v w}_t, {\m A}_t)$ to get $({\v{\hat w}}, {\m{\hat A}})$.  
Suppose  we have the history $({\v{w}}_{t-k}, {\m{A}}_{t-k}), \ldots,$ $({\v{w}}_{t}, {\m{A}}_{t})$.
For $i < t$ define $\gD \v w_{i} = \v w_{i+1} - \v w_{i}$, $\gD \m A_i = \v A_{i+1} - \v A_{i}$ and $\gD (\nabla f^{[d]})_i = \nabla f^{[d]}(\v{w}_{i+1}, \m{A}_{i+1}) - \nabla f^{[d]}(\v{w}_{i}, \m{A}_{i})$, and for $i = t$ define $\gD \v w_{t+1} = {\v{\hat w}} - \v w_{t}$ etc.  
The gradients $\nabla f^{[d]}$ are evaluated implicitly (see \cref{apx:deriv}).
A linear least squares is then solved:

\begin{equation}\label{eq:AA}
\v c^*  = \, \argmin_{{\v c} \in \mathbb{R}^{k+1}} \| \nabla f^{[d]}(\v{\hat{w}}, \m{\hat{A}}) - \sum\nolimits_{i=0}^k c_{t-i} \gD(\nabla f^{[d]})_{t-i}\|^2.
\end{equation}

\vspace{-0.5em}

\noindent We propose $(\v{\tilde{w}}, \m{\tilde{A}}) 
=
(\v{\hat{w}}, {\m{\hat{A}}}) - \sum\nolimits_{i = 0}^k c_{t-i}^* (\gD \v w_{t-i}, \gD \m A_{t-i})$ as a possible next step.
As suggested by \cite{sterck2012nonlinear},
it is checked if $(\v{\tilde{w}}, \m{\tilde{A}}) - (\v{\hat{w}}, \m{\hat{A}})$ and $-\nabla f^{[d]}(\v{\hat{w}}, \m{\hat{A}})$ are sufficiently aligned. 
If so, we do a line search on $f^{[d]}$ from $(\v{\hat{w}}, \m{\hat{A}})$ towards $(\v{\tilde{w}}, \m{\tilde{A}})$ to obtain $(\v{w}_{t+1}, \m{A}_{t+1})$, with the constraint that $\v{w}_{t+1}$ stays in $\gD_{r-1}$.  
Otherwise, we put $(\v{w}_{t+1}, \m{A}_{t+1}) = (\v{\hat{w}}, \m{\hat{A}})$ and restart AA (i.e. clear the history).
See \cref{alg:aa}.

\begin{algorithm}
    \begin{algorithmic}[1]
        \caption{Anderson step for mean matrix $\m A$}
            \label{alg:aa}
            \Function{AndersonStep}{}
            \State Compute ALS update $(\v{\hat{w}}, \m{\hat{A}})$ from current state $(\v w_t, \m A_t)$
            \State Compute gradient $(d\v w, d\m A) =  \nabla f^{[d]}(\v{\hat{w}}, \m{\hat{A}})$ using \cref{apx:deriv}
            \State Solve  \eqref{eq:AA} for $\v c^*$ and compute $(\v{\tilde{w}}, \m{\tilde{A}})$ 
            \State Set search direction $\v u \gets (\v{\tilde{w}}, \m{\tilde{A}}) - (\v{\hat{w}}, \m{\hat{A}})$
            \If{the inner product (after normalization) between $\v u$ and $-(d\v w, d\m A)$ exceeds a user-specified constant $\eps_{AA}$}
                \State Do a line search in $\v u$ direction, and \textbf{return}  mean matrix from search 
            \Else
                \State Use $\m{\hat{A}}$ as the next mean and restart AA
            \EndIf
            \EndFunction
    \end{algorithmic}
\end{algorithm}

\subsection{A Drop-One Procedure} \label{apx:dropone}
We suggest a procedure similar in spirit to dropout from machine learning \cite{hinton2012improving} to aid the nonconvex optimization.
Due to the identifiability guarantee \cref{thm:wMid}, our cost function
\begin{equation*}
    f^{[d]} = \gt_1 f^{(1)} + \ldots + \gt_d f^{(d)}
\end{equation*}
will have the same global minima for almost all $\v \gt$ hyperparameters, when $d$ is large enough and $p$ goes to infinity.  
Thus, we use different choices of $\v \gt$ to add randomness to the descent to avoid attraction to bad local minima.
Specifically we zero out a single $\gt_i$ before each sweep of the mean matrix early on in the descent.  
  We compute
\begin{equation} \label{eq:i-to-drop}
i^* = \argmax_{i \in [d]} \, \big{\|} \! \sum\nolimits_{j \neq i} \gt_j \nabla f^{(j)} \big{\|}
\end{equation}
at the current iterate $({\v w}, {\m A})$, then set $\gt_{i^*} = 0$ and do one ALS sweep of $\m A$.  
We restore $\gt_{i^*}$ to its original value, update $\v w$, and repeat.  
See \cref{alg:drop1}.

In experiments we find that the drop-one procedure indeed discourages convergence to bad local minima.
However when we are close to the global minimum, it increases the number of iterations needed to converge, since it disregards the $i^*$th moment.  
Therefore, we only activate \cref{alg:drop1}  during a warm-up stage.

\begin{algorithm}
    \begin{algorithmic}[1]
        \caption{Drop-one procedure}
            \label{alg:drop1}
            \Function{DropOne}{}
            \State $\v J_s \gets \gt_s \nabla f^{(s)}$ for $s = 1,\ldots,d$
            \State $g \gets \big{\|}\sum_{s \in [d]} \v J_s \big{\|}$
            \State $g^{(i)} \gets \big{\|}\sum_{s \in [d], s \neq i} \v J_s\big{\|}$ for $i = 1,\ldots,d$
            \State $i^* \gets \argmax_{i \in [d]} g^{(i)}$
            \If{$g^{(i^*)} > g$}
                \State In the next iteration, set $\gt_{i^*} = 0$
            \EndIf
            \EndFunction
    \end{algorithmic}
\end{algorithm}

\subsection{Blocked ALS Sweep} \label{apx:blockALS}

We suggest using a blocked ALS sweep on the mean matrix $\m A$.  
Instead of updating one row of $\m A$ per ALS solve as in \cref{alg:Mstep}, we simultaneously update a block $\mathcal{B} \subset [n]$ of $m$ rows per solve, with rows  $[n]\setminus \mathcal{B}$ held fixed.  
We call this \textsc{UpdateMeanBlock}.
It replaces the row $(\v a^k)^{\top}$ in \textsc{UpdateMean} (\cref{alg:Mstep}) by the submatrix $\m A^{\mathcal{B}}$. 
We prepare the normal equations for multiple row solves at once through a blocked variant of \cref{alg:prepnormeqn}.  
We give the details in \cref{alg:block} below, where $\v \pi$ is a matrix, with each column representing a scaling vector $\v v^k/p$, $k \in \cB$,
and $\v r$ is also a matrix, with each column representing the corresponding right-hand side for a row.
The coefficients $\m L$ are shared across the linear systems and only prepared once.

\begin{algorithm}
    \begin{algorithmic}[1]
        \caption{Blocked ALS update on mean matrix $\m A$}
            \label{alg:block}
            \Function{UpdateMeanBlock}{$\{\m G^{A, A}_s\}_{s = 1}^d$, $\{\m G^{A, V}_s\}_{s = 1}^d$, rows to update $\mathcal{B}$, order $d$, hyperparameters $\v \gt$}

            \State $\v \mu \gets \frac{1}{p}\sum_\ell \v v_\ell$
            \State $\v \mu_{\mathcal{B}} \gets \v \mu(\mathcal{B})$, $\m A^{\mathcal{B}} \gets \m A(\mathcal{B}, :)$, $\m V^{\mathcal{B}} \gets \m V(\mathcal{B}, :)$
            \For{$s = 1, ..., d-1$}
                \Comment \textbf{block deflation}
                \State 
                $\m G^{A, A}_s \gets \m G^{A, A}_s - {(}(\m A^{\mathcal{B}})^{*s}{)}^{\! \top} ((\m A^{\mathcal{B}})^{*s})$
                \State 
                $\m G^{A, V}_s \gets \m G^{A, V}_s - ((\m A^{\mathcal{B}})^{*s})^\top((\m V^{\mathcal{B}})^{*s})$
            \EndFor
            \State $\m L, \v r \gets {\small\prepnormeqn(\{\m G^{A, A}_s\}_{s = 1}^d,\{\m G^{A, V}_s\}_{s = 1}^d,(\m V^{\mathcal{B}}/p)^{\! \top},d-1,(2\gt_2 \ldots, d\gt_d))}$
            \State $\m L \gets \m L + \gt_1$
            \Comment \textbf{add  first order cost}
            \State $\v r \gets \v r + \gt_1 \cdot \ind{}\v\mu_{\mathcal{B}}^\top$
            \State Solve for new block $\m A^{\mathcal{B}} \gets \left((\m L^{-1} \v r) / \v w\right)^\top$\!, $\m A(\mathcal{B},:) \gets \m A^{\mathcal{B}}$
            \For{$s = 1, ..., d-1$}
                \Comment \textbf{update $G$ matrices}
                \State 
                $\m G^{A, A}_s \gets \m G^{A, A}_s + ((\m A^{\mathcal{B}})^{*s})^\top((\m A^{\mathcal{B}})^{*s})$
                \State 
                $\m G^{A, V}_s \gets \m G^{A, V}_s + ((\m A^{\mathcal{B}})^{*s})^\top((\m V^{\mathcal{B}})^{*s})$
            \EndFor
             \vspace{1ex}
            \EndFunction
        \Return $\m A, \{\m G^{A, A}_s\}_{s = 1}^d, \{\m G^{A, V}_s\}_{s = 1}^d$
    \end{algorithmic}
\end{algorithm}

We apply a random row-permutation on $\m A$ before each blocked sweep.
Compared to the row-wise ALS sweep, the blocked ALS  prepares and factorizes $\tfrac{1}{m}$ as many $\m L$ matrices during a full sweep of $\m A$, 
while keeping the other computational costs the same.  
The loop depth of the sweep is also reduced by a factor of $m$.
However the blocked update throws out some equations, e.g. when $m=2$ and the first two rows are updated, it ignores differences between entries $\widehat{\t M}^d_{12...}$ and $\t M^d_{12...}$.  
Thus, we only use it in the warm-up stage.

\subsection{Restricting the Parameters} \label{apx:restr_param}
During the warm-up stage, we constrain the parameters to avoid pathological descents.  We lower-bound the weights via
\begin{equation} \label{eq:w-lower}
    (w_i)_t \geq q_t/r\, \,\,\,\, \forall \, i\in [r],
\end{equation}
where $q_t$ is chosen to be a mild bound like $0.1$ for initial $t$ and $0$ for large $t$.
This becomes an additional inequality in the QP problem \eqref{eq:qp} when solving for the weights. 
For the mean matrix, after each blocked ALS update we project the new block $\m A^{\mathcal{B}}$ to the box circumscribing the data vectors $\m V$.

\subsection{Derivative Evaluation} \label{apx:deriv}

If the gradient is evaluated from \eqref{eq:implcost_esp} and the Newton-Gerard formula, it takes $\cO(npr)$ storage, unless a for-loop of depth $n$ is used to reduce the storage to $\cO(pr)$.
However we can compute the derivative with only $\cO(pr)$ extra storage using a for-loop of depth equal to the number of integer partitions of $d$, which is better for $d$ small.

The idea is to rewrite
\eqref{eq:esp2power} as the determinant of a lower Hessenberg matrix \cite{macdonald1998symmetric}:
\begin{equation} \label{eq:esp2det}
    e_d = \frac1{d!} 
    \begin{vmatrix}
        p_1     & 1       & 0      & \cdots & 0      & 0 \\ 
        p_2     & p_1     & 2      & 0      & \cdots & 0 \\ 
        \vdots  & \vdots  & \ddots & \ddots & \ddots & \vdots\\ 
        p_{d-2} & p_{d-3} & \cdots & p_1    & d-2    & 0   \\
        p_{d-1} & p_{d-2} & \cdots & p_2    & p_1    & d-1 \\ 
        p_d     & p_{d-1} & \cdots & p_3    & p_2    & p_1 
    \end{vmatrix}.
\end{equation}
We then expand the determinant as a sum of monomials of power sums:
\begin{equation} \label{eq:Nsformula}
    e_d = \frac{1}{d!} \sum\nolimits_{\v \gl \vdash d} N_{\v \gl} \prod\nolimits_{s \in \v \gl}p_s.
\end{equation}
From an induction on $d$, the monomials correspond to the integer partitions of $d$, and the constant $N_{\v \gl}$ come from the constants in the superdiagonal of \eqref{eq:esp2det}.
This yields another cost evaluation formula.
\begin{proposition} \label{prop:implcost}
    The cost function $f^{[d]}$ \eqref{eq:implcost_esp} can be written as
    \begin{equation} \label{eq:implcost}
        f^{[d]}(\v w, \m A; \v \pi, \m V) = \sum\nolimits_{i = 1}^d \gt_i \sum\nolimits_{\v \gl \vdash i}N_{\v \gl}\ip{\v w}{\m G_{\v \gl}^{A, A} \v w - 2 \m G_{\v \gl}^{A, V} \v \pi} + C,
    \end{equation}
    where $C$ is a constant independent of $\v w$ and $\m A$ that only depends on $\v \pi$ and $\m V$, $N_{\v \gl}$ are combinatorial constants, and $\m G_{\v \gl}^{A, A} \in \mathbb{R}^{r \times r}$ and $\m G_{\v \gl}^{A, V} \in \mathbb{R}^{r \times p}$ are given by (see \eqref{eq:defG} for definitions of $\m G_s$)
    \begin{equation*}
        \m G_{\v \gl}^{A, A} = 
            {{\mathlarger{\ast}}}_{s \in \v \gl} \m G_s^{A, A}
            , \quad \quad 
        \m G_{\v \gl}^{A, V} = {\mathlarger{\ast}}_{s \in \v \gl} 
        \m G_s^{A, V}.
    \end{equation*}
\end{proposition}

\begin{proof}
    Without loss of generality, assume  $i = d$.
    Using \eqref{eq:implcost_esp} and \eqref{eq:Nsformula}, by absorbing all constants into $N_{\v \gl}$ it suffices to show that for any $\v \gl \vdash d$ that
    \begin{equation*}
        \prod_{s \in \v \gl} p_s(\v a_i * \v a_j) = \left(\m G_{\v \gl}^{A, A}\right)_{ij}, \quad \quad 
        \prod_{s \in \v \gl} p_s(\v a_i * \v v_j) = \left(\m G_{\v \gl}^{A, V}\right)_{ij}.
    \end{equation*}
    However the first equation directly from the definition of the $\m G$ matrices,
    since
    \begin{equation*}
        p_s(\v a_i * \v a_j) = \big{(}\m G_s^{A, A}\big{)}_{ij}.
    \end{equation*}
   The second equation follows similarly.
\end{proof}

The derivative with respect to the weight $\v w$ is now clear, and it is sufficient to derive formulas for $d\m G_{\v \gl}^{A, A}/d \m A$ and $d\m G_{\v \gl}^{A, V}/ d\m A$. 
The following lemma is handy.

\begin{lemma} \label{prop:aderiv}
    Let $\v \gl = (\gl_1,\ldots, \gl_{\ell})$ and $\m G_{\v \gl}^{X, Y} = \mathlarger{\ast}_{s = 1}^{\ell}\m G_{\gl_s}^{X, Y}$.  
    Let $\m \gS = \Diag{\v \gs}$ and $\m \gC = \Diag{\v \gc}$.  
    Then 
    \begin{equation*} 
        \pd{}{\m X} \big{\langle} \v \gs, \m G_{\v \gl}^{X,Y} \v \gc\big{\rangle}
        =
        \sum\nolimits_{s = 1}^{\ell} \gl_s\big{(}\eprod{\m X}{\gl_s - 1}\m \gS\big{)}
        \ast 
        \Big{[} 
            \big{(}\eprod{\m Y}{\gl_s}\m\gC\big{)}
            \cdot
            \big{(}\mathlarger{\ast}_{\substack{m = 1,\ldots, \ell\\ m\neq s}}(\eprod{\m Y}{\gl_m})^\top (\eprod{\m X}{\gl_m})\big{)}
        \Big{]}
    \end{equation*}
    and
    \begin{equation*}
        \pd{}{\m X}\big{\langle}\v \gs, \m G_{\v \gl}^{X, X} \v \gs \big{\rangle}
        =
        2\sum\nolimits_{s = 1}^{\ell} \gl_s\big{(}\eprod{\m X}{\gl_s - 1}\m \gS\big{)}
        \ast 
        \Big{[} 
            \big{(}\eprod{\m X}{\gl_s}\m\gS\big{)}
            \cdot
            \big{(}\mathlarger{\ast}_{\substack{m = 1,\ldots, \ell \\ m\neq s}}(\eprod{\m X}{\gl_m})^\top (\eprod{\m X}{\gl_m})\big{)}
        \Big{]}
    \end{equation*}
    with the conventions  $\mathlarger{\ast}_{m \in \emptyset} \, \m M = \ind{} = \eprod{\m M}{0}$.
\end{lemma}

\begin{proof}
    Note that
    \begin{align*}
        \pd{}{\v x_k}\big{\langle}\v \gs, \m G_{\v \gl}^{X, Y} \v \gc \big{\rangle}
        &=
        \pd{}{\v x_k}\sum\nolimits_{i,j}\big{\langle}\gs_i \v x_i^{\ast \gl_1} \otimes \cdots \otimes \v x_i^{\ast \gl_{\ell}}, {\gc_j {\v y_j}^{\! \ast\gl_1} \otimes \cdots \otimes \v y_j^{\ast \gl_{\ell}}} \big{\rangle}\\
        &=
        -2\gs_k\pd{}{\v x_k} 
        \sum\nolimits_j \gc_j \prod\nolimits_{m = 1}^{\ell} \big{\langle}{\v y_j}^{\ast \gl_m}, {\v x_k}^{\ast \gl_m}\big{\rangle}
    \end{align*}
    The lemma now follows from the product rule and the fact that $\m G_{\v \gl}^{X, Y} = (\m G_{\v \gl}^{Y, X})^\top$.
\end{proof}

With these formulas, we can compute the derivatives implicitly with the claimed storage.
In flops the dominating costs are matrix multiplications, and the complexity is  $O(npr)$. 
The coefficients $N_{\v \gl}$  can be precomputed;  for $d \leq 5$ they are
\begin{itemize}
    \item $d=2$: $N_{1^2} = 1$, $N_{2} = -1$;
    \item  $d=3$: $N_{1^3} = 1$, $N_{21} = -3$, $N_3 = 2$;
    \item $d=4$: $N_{1^4} = 1$, $N_{21^2} = -6$, $N_{2^2} = 3$, $N_{31} = 8$, $N_4 = -6$;
    \item  $d=5$: $N_{1^5} = 1$, $N_{21^3} = -10$, $N_{31^2} = 20$, $N_{2^21} = 15$, $N_{41} = -30$, $N_{32} = -20$, $N_5 = 24$.
\end{itemize}

\section{Hyperparameters for Numerical Experiments} \label{apx:hyper}

We list the other hyperparameters used in \cref{sec:exp} in \cref{tab:setup}.

\begin{table}[H] 
    \begin{center}
        \begin{tabular}{|l|l|}
            \hline
            Item                         & Description                                \\
            \hline
            
            Lower bound $q$ on $\v w$        & $0.1 \times 1/r$                           \\
            
            AA depth and $\eps_{AA}$     & 15 and $10^{-4}$                           \\
            Max iteration                & 200                                        \\ 
            Warm up step                 & 20                                         \\
            Block size in warm up        & 2                                          \\
            Initial weight               & vector of $1/r$                            \\ 
            Initial mean                 & standard Gaussian                          \\
            \hline
            \end{tabular}
    \caption{Hyperparameters in the numerical experiments.}
    \label{tab:setup}
    \end{center}
\end{table}

\end{document}